%% file: elsarticle-template-num21.tex
\documentclass[preprint,12pt]{elsarticle}

\usepackage{amssymb}

\journal{System \& Control Letters}

\input{myIncludes}

\begin{document}

\begin{frontmatter}

%% Title, authors and addresses

%% use the tnoteref command within \title for footnotes;
%% use the tnotetext command for theassociated footnote;
%% use the fnref command within \author or \address for footnotes;
%% use the fntext command for theassociated footnote;
%% use the corref command within \author for corresponding author footnotes;
%% use the cortext command for theassociated footnote;
%% use the ead command for the email address,
%% and the form \ead[url] for the home page:
\title{A Geodesic Feedback Law to Decouple \\the Full and Reduced Attitude\tnoteref{label1}}
\tnotetext[label1]{This research was supported in part by the Swedish Foundation for Strategic Research and by the Royal Swedish Academy of Sciences.}
\author{Johan Markdahl$^{*,1}$, Jens Hoppe$^2$, Lin Wang$^3$, Xiaoming Hu$^2$}
%\corref{cor1}
%\fnref{label2}
%\ead{markdahl@kth.se, hoppe@kth.se, hu@kth.se,wanglin@sjtu.edu.cn}
%\ead[url]{https://people.kth.se/~markdahl/}
%\fntext[label2]{Department of Mathematics, KTH Royal Institute of Technology, Stockholm Sweden}
\cortext[cor1]{Corresponding author. \textit{Email address:} \texttt{markdahl@kth.se}.}
\address{${}^1$Luxembourg Centre for Systems Biomedicine, University of Luxembourg, Belval, Luxembourg\\
${}^2$Department of Mathematics, KTH Royal Institute of Technology, Stockholm, Sweden\\
${}^3$Department of Automation, Shanghai Jiao Tong University, Shanghai, China}
%\fntext[label3]{Department of Automation, Shanghai Jiao Tong University, Shanghai, China}

%% use optional labels to link authors explicitly to addresses:
%% \author[label1,label2]{}
%% \address[label1]{}
%% \address[label2]{}

\begin{abstract}
This paper presents a novel approach to the problem of almost global attitude stabilization. The reduced attitude is steered along a geodesic path on the $n-1$-sphere. Meanwhile, the full attitude is stabilized on $\SO$. This action, essentially two maneuvers in sequel, is fused into one smooth motion. Our algorithm is useful in applications where stabilization of the reduced attitude takes precedence over stabilization of the full attitude. A two parameter feedback gain affords further trade-offs between the full and reduced attitude convergence speed. The closed loop kinematics on $\SOT$ are solved for the states as functions of time and the initial conditions, providing  precise knowledge of the transient dynamics. The exact solutions also help us to characterize the asymptotic behavior of the system such as establishing the region of attraction by straightforward evaluation of limits. The geometric flavor of these ideas is illustrated by a numerical example.
\end{abstract}

\begin{keyword}
Attitude control, reduced attitude, geodesics, exact solutions, special orthogonal group.	

%% keywords here, in the form: keyword \sep keyword

%% PACS codes here, in the form: \PACS code \sep code

%% MSC codes here, in the form: \MSC code \sep code
%% or \MSC[2008] code \sep code (2000 is the default)

\end{keyword}

\end{frontmatter}

%\linenumbers

\section{Introduction}

\noindent The attitude tracking problem for a rigid-body is well-known in the literature. It is interesting from a theoretical point of view due to the nonlinear state equations and the topology of the underlying state space $\SOT$. Application oriented approaches to attitude control often make use of parameterizations such as Euler angles or unit quaternions to represent $\SOT$. The choice of parameterization is not without importance since it may affect the limits of control performance \citep{mayhew2011quaternion,chaturvedi2011rigid,bhat2000topological}. An often cited result states that global stability cannot be achieved on $\SOT$ by means of a continuous, time-invariant feedback \citep{bhat2000topological}. It is however possible to achieve almost global asymptotic stability through continuous time-invariant feedback \citep{chaturvedi2011rigid,sanyal2009inertia}, almost semi-global stability \citep{LeeSC}, or global stability by means of a hybrid control approach \citep{mayhew2011quaternion-based}. These subjects have also been studied with regards to the reduced attitude, \ie on the $2$-sphere \citep{bullo1995control,chaturvedi2011rigid}. The problem of pose control on $\SET$ is strongly related to the aforementioned problems. Many of the previously referenced results can be combined with position control algorithms in an inner-and-outer-loop configuration to achieve pose stabilization \citep{roza2012position}.

Like \citep{LeeSC,mayhew2011quaternion-based,chaturvedi2011rigid,sanyal2009inertia}, this paper provides a novel approach to the attitude stabilization problem. The generalized full attitude is stabilized on $\SO$. Meanwhile, the generalized reduced attitude is steered along a geodesic path on the $(n-1)$-sphere. The motion of the reduced attitude is decoupled from the remaining degree of freedom of the full attitude but not vice versa. An action consisting of two sequential manoeuvres is thus fused into one smooth motion. This algorithm is of use in applications where the stabilization of the reduced attitude takes precedence over that of the full attitude. A two parameter feedback gain affords further trade-offs regarding the full and reduced attitude convergence speed. The kinematic model is suited for applications in the field of visual servo control \citep{chaumette2006visual,chaumette2007visual}. Consider a camera that is tracking an object. The goal is to keep the camera pointing towards the object whereas the roll angle is of secondary importance. The proposed algorithm solves this problem by steering the principal axis directly towards the object while simultaneously stabilizing the roll angles without resorting to a non-smooth control consisting of two separate motions.

While literature on the kinematics and dynamics of $n$-dimensional
rigid-bodies (\eg \citep{hurtado2004hamel}) may primarily be theoretically motivated, the developments also provide a unified framework for the cases of $n\in\{2,3\}$. The generalized reduced attitude encompasses all orientations in physical space: the heading on a circle, the reduced attitude on the sphere, and the unit quaternions on the $3$-sphere. Relevant literature includes works concerning stabilization \citep{maithripala2006almost}, synchronization \citep{lageman2009synchronization}, and estimation \citep{lieobserver} on $\SO$. It also includes the previous work \citep{markdahl2014analytical,markdahl2013analytical} of the authors. Note that work on $\SO$ for $n\geq4$ is not only of theoretical concern; it also finds applications in the visualization of high-dimensional data \citep{thakur2008}. 

Exact solutions to a closed-loop system yields insights into both its transient and asymptotic behaviors and may therefore be of value in applications. The literature on exact solutions to attitude dynamics may, roughly speaking, be divided into two separate categories. First, there are a number of works where the exact solutions are obtained during the control design process, \eg using exact linearization \citep{dwyer1984exact}, optimal control design techniques such as the Pontryagin maximum principle \citep{spindler1998optimal}, or in the process of building an attitude observer \citep{attitudeobserver}. Second, there are studies of the equations defining rigid-body dynamics under a set of specific assumptions whereby the exact solutions become one of the main results \citep{elipe2008exact,doroshin2012exact,ayoubi2009asymptotic}. This paper belongs to the second category. The closed-loop kinematics on $\SOT$ are solved for the states as functions of time and the initial conditions, providing  precise knowledge of the workings of the transient dynamics.

Recent work on the problem of finding exact solutions to closed-loop systems on $\SO$ includes \citep{markdahl2014analytical,markdahl2013analytical}. Related but somewhat different problems are addressed in  \citep{elipe2008exact,doroshin2012exact,ayoubi2009asymptotic}. Earlier work \citep{markdahl2012exact} by the authors is strongly related but also underdeveloped; its scope is limited to the case of $\mathrm{SO}(3)$. This paper concerns a generalization of the equations studied in \citep{markdahl2012exact,markdahl2013analytical}. The results of \citep{markdahl2013analytical} is also generalized in \citep{markdahl2014analytical}, partly towards application in model-predictive control and sampled control systems and without focus on the behavior of the reduced attitude. The work \citep{markdahl2015automatica} addresses the problem of continuous actuation under discrete-time sampling. The exact solutions provide an alternative to the zero-order hold technique. The algorithm alternates in a fashion that is continuous in time between the closed-loop and open-loop versions of a single control law. The feedback law proposed in this paper can also be used in such applications by virtue of the exact solutions.

\section{Preliminaries}
\label{secB:prel}

\noindent Let $\ma{A}, \ma{B}\in\C^{n\times n}$. The spectrum of $\ma{A}$ is written as $\sigma(\ma{A})$. Denote the transpose of $\ma{A}$ by $\ma{A}^{\!\top}$ and the complex conjugate by $\ma{A}^*$.  The inner product is defined by $\langle\ma{A},\ma{B}\rangle=\trace(\ma{A}^{\!\top}\ma{B})$ and the Frobenius norm by $\|\ma{A}\|_F=\langle \ma{A},\ma{A}\rangle^{\frac12}.$ The outer product of two vectors $\ve{x},\ve{y}\in\R^n$ is defined as $\ve{x}\otimes\ve{y}=\ve{x}\vet{y}$. The commutator of two matrices is $[\ma{A},\ma{B}]=\ma{AB}-\ma{BA}$ and the anti-commutator is $\{\ma{A},\ma{B}\}=\ma{AB}+\ma{BA}$.

 The special orthogonal group is $\SO=\{\ma{R}\in \R^{n\times n}\,|\,\ma{R}\inv=\mat{R}\hspace{-1.3mm},\hspace{1.3mm}\,\det\ma{R}=1\}$. The special orthogonal Lie algebra is $\so=\{\ma{S}\in\R^{n\times n}\,|\,\mat{S}=-\ma{S}\}$. The $n$-sphere is $\Sn=\{\ve{x}\in\R^{n+1}\,|\,\|\ve{x}\|=1\}$. The geodesic distance between $\ve{x},\ve{y}\in\Sn$ is given by $\vartheta(\ve{x},\ve{y})=\arccos\langle\ve{x},\ve{y}\rangle$. An almost globally asymptotically stable equilibrium is stable and attractive from all initial conditions in the state space except for a set of zero measure. The terms attitude stabilization, reduced attitude stabilization, and geodesic path refer to the stabilization problem on $\SO$, the $n$-sphere, and curves that are geodesic up to parametrization respectively.

%The terms attitude stabilization and reduced attitude stabilization are used in somewhat more general context than what is conventional, we take them to refer to the stabilization problem on $\SO$ and the $n$-sphere respectively. The term geodesic is used in a relaxed fashion to refer to curves that are geodesic up to parametrization.

%For the purposes of this paper, it is convenient to work with extended hyperbolic functions. 

Real matrix valued, real matrix variable hyperbolic functions are defined by means of the matrix exponential, \eg $\cosh:\R^{n\times n}\rightarrow\GL$ is given by $\cosh(\ma{A})=\smash{\frac12}[\exp(\ma{A})+\exp(-\ma{A})]$ for all $\ma{A}\in\R^{n\times n}$. Let $\Log:\C \backslash\{0\}\rightarrow\C$ denote the principal logarithm, \ie $\Log z=\log r+i\vartheta$, where $z=r\e^{i\vartheta}$ and $\vartheta\in(-\pi,\pi]$. Let $\Atanh:\C\backslash\{-1,1\}\rightarrow\C$ denote the principal inverse hyperbolic tangent, \ie $\Atanh z=\tfrac12[\Log(1+z)-\Log(1-z)]$. Note that $\tanh:\C\backslash\{-1,1\}\rightarrow\C$ satisfies $\tanh\Atanh z=z$ for all $z\in\C \backslash\{-1,1\}$. Extend these definitions to the extended real number line $\R\cup\{-\infty,\infty\}$ and the Riemann sphere $\C\cup\{\infty\}$ by letting $\log 0=-\infty$, $\Atanh 1=\infty$, $\tanh\infty=1$ \etc \citep{rudin1987real}.

%Finally, denote the $n$-sphere by $\mathcal{S}^n$.

\section{Problem Description}

\subsection{Stabilization and Tracking}

\noindent The orientation or attitude of a rigid body is represented by a rotation matrix that transforms the body fixed frame into a given inertial fixed frame. Let $\ma{X}\in\mathrm{SO}(3)$ denote this rotation matrix. The kinematics of a rigid body dictates that $\md{X}=\ma{\Omega}\ma{X}$, where $\ma{\Omega}\in\mathrm{so}(3)$ is a skew-symmetric matrix representing the angular velocity vector of the rigid body. The attitude stabilization problem is the problem of designing a feedback law that stabilizes a desired frame $\ma[d]{X}$ which without loss of generality can be taken to be the identity matrix.

The attitude tracking problem concerns the design of an $\ma{\Omega}$ that rotates $\ma{X}$ into a desired moving frame $\ma[d]{X}\in\SOT$. Assume that $\ma[d]{X}$ is generated by $\md[d]{X}=\ma[d]{\Omega}\ma[d]{X}$, where $\ma[d]{\Omega}\in\sot$ is known. Furthermore assume that the relative rotation error $\ma{R}=\mat[d]{X}\ma{X}\in\SOT$ is known to the feedback algorithm. Note that rotating $\ma{X}$ into $\ma[d]{X}$ is equivalent to rotating $\ma{R}$ into $\ma{I}$. Moreover,
\begin{align}
\md{R}&=\mdt[d]{X}\ma{X}+\mat[d]{X}\md{X}=(\ma[d]{\Omega}\ma[d]{X})\mtr\ma{X}+\mat[d]{X}\ma{\Omega}\ma{X}=-\mat[d]{X}\ma[d]{\Omega}\ma[d]{X}\ma{R}+
\mat[d]{X}\ma{\Omega}\ma[d]{X}\ma{R}\nonumber\\
&=\mat[d]{X}(-\ma[d]{\Omega}+\ma{\Omega})\ma[d]{X}\ma{R}=\ma{U}\ma{R},\label{eqB:Rd}
\end{align}
where  $\ma{U}=\mat[d]{X}(-\ma[d]{\Omega}+\ma{\Omega})\ma[d]{X}\in\sot$. The kinematic level attitude tracking problem in the case of known $\ma[d]{R},\ma[d]{\Omega}$ can hence be reduced to the attitude stabilization problem. It is also clear that attitude stabilization is a special case of attitude tracking.%, \ie the problem of specifying an $\ma{U}$ in terms of $\ma{R}$ such that the identity matrix is an asymptotically stable equilibrium of \eqref{eq:Rd}.

From a mathematical perspective it is appealing to strive for generalization. Consider the evolution of a positively oriented $n$-dimensional orthogonal frame represented by $\ma{R}\in\SO$. The dynamics are given by
\begin{align}\label{eqB:SOn}
\md{R}&=\ma{U}\ma{R},
\end{align}
where $\ma{U}\in\so$. The kinematic level generalized attitude stabilization problem concerns the design of an $\ma{U}$ that stabilizes the identity matrix on $\SO$. It is assumed that $\ma{R}$ can be actuated along any direction of its tangent plane at the identity $\ts[\SO]{\ma{I}}=\so$. Note that $\SO$ is invariant under the kinematics \eqref{eqB:SOn}, \ie any solution $\ma{R}(t)$ to \eqref{eqB:SOn} that satisfies $\ma{R}(0)=\ma[0]{R}\in\SO$ remains in $\SO$ for all $t\in[0,\infty)$.

\subsection{The Reduced Attitude}

\noindent It is sometimes preferable to only consider $n-1$ of the $\tfrac12n(n-1)$ degrees of freedom on $\SO$. In the case of $\SOT$, these correspond to the reduced attitude \citep{chaturvedi2011rigid}. The reduced attitude consists of the points on the unit sphere $\mathcal{S}^2\simeq\mathsf{SO}(3)/\mathsf{SO}(2)$. It formalizes the notion of pointing orientations such as the two degrees of rotational freedom possessed by objects with cylindrical symmetry. The reduced attitude is also employed in redundant tasks like robotic spray painting and welding that only require the utilization of two of the usual three degrees of rotational freedom in physical space \citep{siciliano2008springer}.

Reduced attitude control by means of kinematic actuation is a special case of control on the unit $n$-sphere, $\Sn=\{\ve{x}\in\R^{n+1}\,|\,\|\ve{x}\|=1\}$. The generalized reduced attitude can be used to model all physical rotations. The heading of a two-dimensional rigid-body is an element of $\GC$, the pointing direction of a cylindrical rigid-body is an element of $\St$, and the full attitude can be parametrized by $\mathcal{S}^3$ through a composition of two maps via the unit quaternions $\mathcal{S}^0(\Ham)=\{q\in\Ham\,|\,|q|=1\}$.

%Note that unlike $\SOT$, $\St$ is not a group. Indeed, the findings of this thesis suggest that good control performance may in a certain sense be easier to obtain on $\St$ as compared to $\SOT$.

%Consider the dynamics on the $n$-sphere. Take any $\ve{x}\in\Sn$. Unlike on $\SO$, the tangent space $\ts[\Sn]{\ve{x}}$ depends on $\ve{x}$. Since $\ve{x}$ is of constant length, it follows that $\vd{x}\perp\ve{x}$. Let $\ve{u}$ denote the input signal. Then 

Let $\ve[1]{e}\in\mathcal{S}^{n-1}$ be a vector expressed in the body-fixed frame of an $n$-dimensional rigid body. The reduced attitude $\ve{x}\in\mathcal{S}^{n-1}$ is defined as the inertial frame coordinates of $\ve[1]{e}$, \ie $\ve{x}=\ma{X}\ve[1]{e}$. The reduced attitude stabilization problem is solved by a feedback algorithm that can turn $\ve{x}$ into any desired value $\ve[d]{x}\in\mathcal{S}^{n-1}$. Note that $\vd{x}=\vd{X}\ve[1]{e}=\ma{\Omega}\ma{X}\ve[1]{e}=\ma{\Omega}\ve{x}$. Assume that $\ve[d]{x}=\ma[d]{X}\ve[1]{e}$ satisfies $\vd[d]{x}=\ma[d]{\Omega}\ve[d]{x}$. Set $\ve{r}=\mat[d]{X}\ve{x}=\mat[d]{X}\ma{X}\ve[1]{e}=\ma{R}\ve[1]{e}$. Turning $\ve{x}$ into $\ve[d]{x}$ is equivalent to turning $\ve{r}$ into $\ve[1]{e}$. Moreover, $\vd{r}=\md{R}\ve[1]{e}=\ma{U}\ma{R}\ve[1]{e}=\ma{U}\ve{r}$, where $\ma{U}=\mat[d]{X}(-\ma[d]{\Omega}+\ma{\Omega})\ma[d]{X}\in\so$,  like in the $\SOT$ case. The evolution of $\ve{r}$ is controllable on $\mathcal{S}^{n-1}$ \citep{brockett1973lie}.

Note that $\vd{r}\perp \ve{r}$ due to $\ma{U}\in\so$. The set $\so$ has more than enough degrees of freedom to fully actuate $\ve{r}$. It suffices to express $\ma{U}$ in terms of a control $\ve{u}\in\R^{n}$ by letting $\ma{U}=\ve{u}\otimes\ve{r}-(\ve{u}\otimes\ve{r})\mtr\in\so$. Then,
\begin{align}\label{eqB:u}
\vd{r}&=\ve{u}\otimes\ve{r}\,\ve{r}-\ve{r}\otimes\ve{u}\,\ve{r}=\ve{u}-\langle\ve{u},\ve{r}\rangle\ve{r}=(\ma{I}-\ve{r}\otimes\ve{r})\ve{u},
\end{align}
where the identity $\ve{u}\otimes\ve{v}\,\ve{w}=\langle\ve{v},\ve{w}\rangle\ve{u}$ for any $\ve{u},\ve{v},\ve{w}\in\R^n$ is used. Since $\ve{u}$ is arbitrary, $\ve{r}$ can be actuated in any direction along its tangent plane $\smash{\ts[\mathcal{S}^{n-1}]{\ve{r}}}=\{\ve{v}\in\R^{n}\,|\,\langle\ve{r},\ve{v}\rangle=0\}$, \ie the hyperplane of vectors orthogonal to $\ve{r}$. The generalized kinematic level reduced attitude stabilization problem concerns the design of an $\ve{u}$ that stabilizes $\ve[1]{e}$.

Note that by setting $\ma{u}=\ve{v}$, the dynamics \eqref{eqB:u} moves $\ve{r}$ in the steepest descent direction of the geodesic distance $\vartheta(\ve{v},\ve{r})=\arccos\langle\ve{v},\ve{r}\rangle$ in the case of a constant $\ve{v}\in\Sn$,
\begin{align*}
\argmin_{\ve{u}\in\Sn}\dot{\vartheta}=\argmax\limits_{\ve{u}\in\Sn}\tfrac{\diff}{\diff t}\langle\ve{v},\ve{r}\rangle=\argmax_{\ve{u}\in\Sn}\langle\ve{v},(\ma{I}-\ve{r}\otimes\ve{r})\ve{u}\rangle=\ve{v}.
\end{align*}
We say that a feedback $\ve{u}$ is geodesic if it controls the system along a path of minimum length in the state-space, \ie if there is a reparametrization of time that turns the state trajectory into a geodesic curve. %The definition of a geodesic is more precise, and imply properties such as constant velocity. %However, such a parameterization would in most cases be undesirable from a control theory perspective since it would imply discontinuous feedback.

\subsection{Problem Statement}

\noindent This paper concerns the formulation and proof of stability of a control law that solves the problem of stabilizing the full attitude almost globally while simultaneously providing a geodesic feedback for the reduced attitude. In other words, we design a control signal $\ma{U}$ such that $\ma{I}$ is an almost globally asymptotically stable equilibrium of the full attitude $\ma{R}$ and the reduced attitude $\ma{r}$ moves towards $\ve[1]{e}$ along a great circle. Moreover, on $\SOT$, which is the most interesting case for application purposes, we also solve the closed-loop equations generated by the proposed algorithm for $\ma{R}$ as a function of time, the initial condition, and two gain parameters.

\section{Control Design}

\noindent This section presents an algorithm that stabilizes the identity matrix on $\SO$. The proposed algorithm is also shown to stabilize the generalized reduced attitude along a geodesic curve on the $(n-1)$-sphere from all initial conditions except a single point.

\begin{algorithm} \label{algo:2014}
Let the feedback $\ma{U}:\SO\rightarrow\so$ be given by
\begin{align}
\ma{U}&=\ma{P}\mat{R}-\ma{R}\ma{P}+k\,\ma{R}\ma{Q}(\mat{R}-\ma{R})\ma{Q}\mat{R},\label{eq:control2014}
\end{align}
where $\ma{P}\in\{\ma{A}\in\R^{n\times n}\,|\,\ma{A}^2=\ma{A},\, \mat{A}=\ma{A}\}$, \ie $\ma{P}$ is a constant orthogonal projection, $k\in(0,\infty)$, and $\ma{Q}=\ma{I}-\ma{P}$. 
\end{algorithm}

\begin{remark}
The control gain $k$ is introduced to afford a trade-off between the reduced and full attitude convergence rates. Note that a second feedback gain parameter can be introduced by multiplying $\ma{U}$ by some positive constant. This is equivalent to scaling time, wherefore a single parameter suffices.% It hence suffices to consider just one parameter.
%As always, multiplying the right-hand side of a differential equation by a gain is simply a matter of scaling time. It therefore suffices to include one parameter. Suppose we scale the input by an additional parameter to maintain a bounded control norm. Lower values of $k$ make the reduced attitude converge faster at the expense of the speed of convergence of the full attitude and vice versa. The geodesic property is maintained since the structures of $\ma{P}$ and $\ma{Q}$ are unchanged.
\end{remark}

The resulting closed loop system is
\begin{align}
\md{R}&=\ma{U}\ma{R}=\ma{P}-\ma{R}\ma{P}\ma{R}+k\,\ma{R}\ma{Q}(\mat{R}-\ma{R})\ma{Q}.\label{eq:closed2014}
\end{align}
Note that $\ma{Q}$ is also an orthogonal projection matrix, \ie $\ma{Q}^2=\ma{Q}$ and $\mat{Q}=\ma{Q}$. Moreover, $\ma{P}$ and $\ma{Q}$ satisfy the relations $\ma{P}+\ma{Q}=\ma{I}$ and $\ma{P}\ma{Q}=\ma{0}$..

Consider the case of $\ma{P}=\ve[1]{e}\otimes\,\ve[1]{e}$, where $\{\ve[1]{e},\ldots,\ve[n]{e}\}$ denotes the standard basis of $\R^n$. This is equivalent, up to a change of coordinates, to the case of $\rank\ma{P}=1$. The dynamics of the reduced attitude are given by
\begin{align} \label{eq:rd}
\vd{r}=\ve[1]{e}-\langle\ve[1]{e},\ve{r}\rangle\ve{r},
\end{align}
\ie equation \eqref{eqB:u} with $\ve{u}=\ve[1]{e}$. This feedback results in $\ve{r}\in\mathcal{S}^{n-1}$ moving towards $\ve[1]{e}$ along a great circle. The case of $\ma{P}=\ve[1]{e}\otimes\ve[1]{e}$ is further explored in Section \ref{sec:reduced} and \ref{sec:exact}.

The first skew-symmetric difference in \eqref{eq:control2014} is designed to steer $\ma{R}\ma{P}$ to $\ma{P}$. If $\ma{P}=\ma{I}$, this control action suffices to stabilize the identity matrix. Otherwise, when $\|\ma{RP}-\ma{P}\|_2$ is sufficiently small, the second skew-symmetric term kicks in to steer $\ma{R}\ve{Q}$ to $\ve{Q}$. Intuitively speaking, in the case of $\ma{P}=\ve[1]{e}\times\ve[1]{e}$, this can be interpreted as stabilization on $\mathcal{S}^{n-1}\simeq\SO\backslash\mathsf{SO}(n-1)$ followed by stabilization on $\mathsf{SO}(n-1)$, where the two control actions have been fused into one smooth motion. Since $\ma{R}=\ma{R}(\ma{P}+\ma{Q})=\ma{R}\ma{P}+\ma{R}\ma{Q}$, if $\lim_{t\rightarrow\infty}\ma{R}(t)\ma{P}=\ma{P}$ and  $\lim_{t\rightarrow\infty}\ma{R}(t)\ma{Q}=\ma{Q}$, then $\lim_{t\rightarrow\infty}\ma{R}(t)=\ma{I}$. 

 %The proof of stability, which makes use of a reduction theorem, mimics this intuitive idea of a two-step control algorithm.

\begin{remark}
The algorithm of \citep{markdahl2012exact} as well as the two algorithms of \citep{markdahl2013analytical} are special cases of Algorithm \ref{algo:2014}. In \citep{markdahl2012exact}, $n=3$ and $\ma{P}=\ve[1]{e}\otimes\ve[1]{e}$. In \citep{markdahl2013analytical}, $n\in\N$, and  $\ma{P}=\ma{I}$ or $\ma{P}=\ma{I}-\ve[n]{e}\otimes\ve[n]{e}$ for the two respective algorithms. This paper explores the general case of  $n\in\N$ and $\ma{P}\in\{\ma{A}\in\R^{n\times n}\,|\,\ma{A}^2=\ma{A},\, \mat{A}=\ma{A}\}$, with focus on  projection matrices that satisfy $\rank\ma{P}\leq n-2$. The case of $\rank\ma{P}\in\{1,n-1,n\}$ is considerably simpler.
\end{remark}

%\begin{lemma}\label{le:PQ}
%Given an orthogonal projection matrix $\ma{P}\in\R^{n\times n}$, \ie $\ma{P}^2=\ma{P}$ and $\ma{P}\mtr=\ma{P}$, let $\ma{Q}=\ma{I}-\ma{P}$. Then $\ma{P}\ma{Q}=\ma{0}$, and $\ma{Q}$ is also an orthogonal projection matrix.
%\end{lemma}
%
%\begin{proof}
%The proof is straightforward.
%\end{proof}

%\begin{remark}
%In the following, frequent use will be made of Lemma \ref{le:PQ}, often without reference. Note that Algorithm \ref{algo:2014} requires the gain matrix $\ma{P}$ to be an orthogonal projection, and that $\ma{Q}$ is uniquely determined by $\ma{P}$. The control \eqref{eq:control2014} can hence be parametrized by any constant orthogonal projection.
%\end{remark}

\section{The Reduced Attitude}
\label{sec:reduced}

%\noindent This section contains proofs of stability in the reduced attitude case and the case where the effect of the second skew-symmetric term in the feedback of Algorithm \ref{algo:2014} is disregarded or removed by setting $k=0$. %The proof of stability in the general case is partly based on the latter result. %That proof makes use of a reduction theorem which, roughly speaking, reduces the full attitude case to the reduced attitude case.
%
%\subsection{The Reduced Attitude}
%
%
%
\noindent Let us show that \eqref{eq:control2014} is a geodesic feedback for the reduced attitude in the special case of $\smash{\ma{P}=\ve[1]{e}\otimes\ve[1]{e}}$, which is equivalent to the case of $\rank\ma{P}=1$ up to a change of coordinates. Our objective is to turn the unit vector $\ve{r}$ into $\ve[1]{e}$ by means of a continuous rotation about a constant axis. A stability proof  in the case of a general orthogonal projection matrix $\ma{P}$ is also given. 

\begin{proposition}\label{propB:reduced}
Set $\ma{P}=\ve[1]{e}\otimes\ve[1]{e}$ in Algorithm \ref{algo:2014}. The equilibrium $\ve{r}=\ve[1]{e}$ of \eqref{eq:rd} is almost globally exponentially stable on $\mathcal{S}^{n-1}$. The unstable manifold $\{-\ve[1]{e}\}$ corresponds to a single point that is antipodal to the desired equilibrium. Moreover, $\ve{r}$ evolves from its initial value to $\ve[1]{e}$ along a great circle.
\end{proposition}

\begin{proof} Define a candidate Lyapunov function $V:\mathcal{S}^{n-1}\rightarrow[0,2]$ by
\begin{align*}
V&=\tfrac12\|\ve{r}-\ve[1]{e}\|^2_2=1-\langle\ve[1]{e},\ve{r}\rangle.
\end{align*}
Then $\dot{V}=-\langle\ve[1]{e},\ve[1]{e}-\langle\ve[1]{e},\ve{r}\rangle\ve{r}\rangle=-(2-V)V$ by \eqref{eq:rd}. The equilibrium $\ve{r}=\ve[1]{e}$ is almost globally asymptotically stable by application of LaSalle's invariance principle and Lyapunov's theorem. Local exponential stability follows from $\dot{V}\leq-V$ on the hemisphere $\{\ve{r}\in\mathcal{S}^{n-1}\,|\,\langle\ve[1]{e},\ve{r}\rangle\geq0\}$.

The Euclidean metric $d(\ve{x},\ve{y})=\|\ve{x}-\ve{y}\|_2$ and intrinsic arc length metric $\vartheta(\ve{x},\ve{y})=\arccos\langle\ve{x},\ve{y}\rangle$ for any $\ve{x},\,\ve{y}\in\mathcal{S}^{n-1}$ are related by $d(\ve{x},\ve{y})^2=
\|\ve{x}-\ve{y}\|^2_2=2(1-\langle\ve{x},\ve{y}\rangle)=2[1-\cos\vartheta(\ve{x},\ve{y})]$ for all $\ve{x},\ve{y}\in\mathcal{S}^{n-1}$. It follows that the gradients of $d$ and $\vartheta$, defined using the metric tensor induced by the inner product in Euclidean space, are negatively aligned. Moreover, $\nabla V=-\ve[1]{e}=-\ve{u}$, so $\ve{r}$ only moves in the steepest decent direction of $\vartheta$, \ie $\ve{r}(t)$ is a geodesic curve up to parametrization which makes $\ve{u}$ a geodesic feedback.\end{proof}

\begin{remark}
The problem of geodesic feedback as well as other control problems on the sphere such as dynamic level control and tracking on the $2$-sphere are addressed in \citep{bullo1995control}. The work \citep{brockett1973lie} explores the problems of controllability, observability, and minimum energy optimal control on the $n$-sphere.
\end{remark}	

Consider the case of a general orthogonal projection matrix $\ma{P}$. Postmultiply $\md{R}$ by $\ma{P}$ to find that
\begin{align}
\md{R}\ma{P}&=\ma{P}-(\ma{R}\ma{P})^2, \label{eq:RP}
\end{align}
where the $k$-term in \eqref{eq:closed2014} is canceled due to $\ma{P}\ma{Q}=\ma{0}$. Note that $\ma{P}$ and $\ma{Q}$ satisfy the following relations $\cosh(\ma{P}t)=\ma{Q}+\cosh(t)\ma{P},\quad
\sinh(\ma{P}t)=\sinh(t)\ma{P}$, where $\cosh$ and $\sinh$ denote matrix variable, matrix valued hyperbolic functions defined by replacing the exponential function in their scalar variable, scalar valued analogues by the matrix exponential. The matrix $\cosh(\ma{P}t)$ is nonsingular since $\sigma[\cosh(\ma{P}t)]\subset\{1,\cosh t\}\subset(1,\infty)$.

\begin{proposition}\label{th:RP} The unique solution to $\md{H}=\ma{P}-\ma{H}^2$ as a trajectory in the homogeneous space 
\begin{align}
\mathcal{H}=\{\ma{H}\in\R^{n\times n}\,|\,\ma{H}=\ma{R}\ma{P},\ma{R}\in\SO\}\label{eq:hspace}
\end{align}
with initial condition $\ma{H}(0)=\ma[0]{H}\in\mathcal{H}$ is given by
\begin{align}
\ma{H}(t)=&{}[\sinh(\ma{P}t)+\cosh(\ma{P}t)\ma[0]{H}][\cosh(\ma{P}t)+\sinh(\ma{P}t)\ma[0]{H}]\inv.\label{eq:H}
\end{align}
\end{proposition}

\begin{proof} 
A proof of global existence and uniqueness is given by Lemma \ref{leB:unique} in Appendix \ref{appB:lemmas}. Denote $\ma{H}(t)=\ma{X}(t)\ma{Y}(t)\inv$, where $\ma{X}(t)=\sinh(\ma{P}t)+\cosh(\ma{P}t)\ma[0]{H}$, $\ma{Y}(t)=\cosh(\ma{P}t)+\sinh(\ma{P}t)\ma[0]{H}$. It can be shown that $\ma{Y}\inv(t)$ is well-defined for all $t\in[0,\infty)$. Note that $\md{X}(t)=\ma{P}\ma{Y}(t)$ and $\md{Y}(t)=\ma{P}\ma{X}(t)$. The proof is by verification that $\ma{H}(t)$ satisfies \eqref{eq:RP}, 
\begin{align*}
\md{H}(t)&=\md{X}(t)\ma{Y}(t)\inv-\ma{X}(t)\ma{Y}(t)\inv\md{Y}(t)\ma{Y}(t)\inv\\
&=\ma{P}\ma{Y}(t)\ma{Y}(t)\inv-\ma{X}(t)\ma{Y}(t)\inv\ma{P}\ma{X}(t)\ma{Y}(t)\inv\\
&=\ma{P}-\ma{H}(t)\ma{P}\ma{H}(t)=\ma{P}-\ma{H}(t)^2.
\end{align*}
Moreover, $\ma{H}(0)=\ma[0]{H}$.\end{proof}

Introduce the set of rotation matrices with partly negative spectrum,
\begin{align*}
\NO=\{\ma{R}\in\SO\,|\,-1\in\sigma(\ma{R})\}.
\end{align*}
Observe that $\NO$ is a set of zero measure in $\SO$. One can show that $\mathcal{N}=\{\ma{R}\in\SOT\,|\,\mat{R}=\ma{R}\}\backslash\{\ma{I}\}$ in the case of $\SOT$, but such a relation does not hold in higher dimensions as illustrated by the matrix
\begin{align*}
\ma{R}=\begin{bmatrix}
\,\ma[11]{R} & \maspace\phantom{-}\ma{0}\\
\ma{0} & \maspace-\ma{I}
\end{bmatrix}\in\SO
\end{align*}
which belongs to $\mathcal{N}$ for all $\ma[11]{R}\in\mathrm{SO}(n-2)$, where $n\geq4$.

\begin{proposition}\label{prop:RPstable}
The system $\md{H}=\ma{P}-\ma{H}^2$ over the homogeneous space  $\mathcal{H}$ given by \eqref{eq:hspace} converges to the equilibrium manifold  $\{\ma{H}\in\mathcal{H}\,|\,\ma{H}^2=\ma{P}\}$. The equilibrium $\ma{H}=\ma{P}$ is almost globally asymptotically stable.
\end{proposition}

\begin{proof}
Consider the candidate Lyapunov function $V=\trace(\ma{P}-\ma{H})$. Since
\begin{align*}
\dot{V}=-\trace(\ma{P}-\ma{H}^2)=-p+\sum_{i=1}^p \lambda_i^2=-p+\sum_{i=1}^p a_i^2-b_i^2,
\end{align*} 
where $p=\rank\ma{P}$ and $\lambda_i=a_i+ib_i$ for $i\in\{1,\ldots,p\}$ are eigenvalues of $\ma{H}$ (the eigenvalue zero has at least algebraic multiplicity $n-p$). Note that $\rho\leq\|\ma{H}\|_2\leq\|\ma{R}\|_2\|\ma{P}\|_2=1$, where $\rho$ is the spectral radius of $\ma{H}$, implies $a_i^2+b_i^2\leq 1$. It follows that
\begin{align*}
\dot{V}\leq-p+\sum_{i=1}^p 1-b_i^2-b_i^2=-2\sum_{i=1}^p b_i^2,
\end{align*} 
which is negative semidefinite. The spectrum of $\ma{H}$ converges to $\{-1,0,1\}$ as time goes to infinity by LaSalle's invariance principle. 

Note that if $(\lambda,\ve{v})$ is an eigenpair of $\ma{H}=\ma{R}\ma{P}$ with $\lambda\in\{-1,1\}$, then $\ma{P}\ve{v}=\ve{v}$, $\ma{R}\ve{v}=\lambda\ve{v}$. Let $\mathcal{V}=\{\ve[1]{v},\ldots,\ve[p]{v}\}$ be a linearly independent set of eigenvectors with $\lambda_i\in\{-1,1\}$   that maximizes $|\mathcal{V}|=p$. Since $\ma{R}$ is normal, there is a basis $\{\ve[1]{v},\ldots,\ve[n]{v}\}$ of $\R^n$ where $(\lambda_i,\ve[i]{v})$ are eigenpairs of $\ma{R}$ with $\lambda_i\in\{-1,1\}$ for $i\in\{1,\ldots,p\}$ and  $\ma{P}\ve[i]{v}=\ve{0}$ for $i\in\{p+1,\ldots,n\}$. Clearly, $(\ma{P}-\ma{H}^2)\ve[i]{v}=\ve{0}$ if $i\in\{p+1,\dots,n\}$ and
\begin{align*}
(\ma{P}-\ma{H}^2)\ve[i]{v}=(1-\lambda^2)\ve{v}=\ve{0}
\end{align*}
if $i\in\{1,\ldots,p\}$. The matrix $\ma{P}-\ma{H}^2$ maps a basis of $\R^n$ to zero, and is therefore zero. It follows that $\ma{R}\ma{P}$ converges to $\{\ma{H}\in\mathcal{H}\,|\,\ma{H}^2=\ma{P}\}$. 

Lemma \ref{le:PRP} in Appendix \ref{appB:lemmas} tells us that $-1\notin\sigma(\ma[0]{R})$ implies $-1\notin\sigma(\ma{P}\ma[0]{R}\ma{P})$. Suppose $\ma[0]{R}\in\SO\backslash\NO$. Proposition \ref{th:RP} and some calculations yield
\begin{align*}
\lim_{t\rightarrow\infty}\ma{P}\ma{R}(t)\ma{P}={}&\lim_{t\rightarrow\infty}\ma{P}[\tanh (t)\,\ma{I}+\ma{P}\ma[0]{R}\ma{P}][\ma{I}+\tanh(t)\ma{P}\ma[0]{R}\ma{P}]\inv=\ma{P},\\
\lim_{t\rightarrow\infty}\ma{Q}\ma{R}(t)\ma{P}={}&\lim_{t\rightarrow\infty}\ma{Q}\ma[0]{R}\ma{P}[\cosh(\ma{P}t)+\sinh(\ma{P}t)\ma[0]{R}\ma{P}]\inv=\ma{0},
\end{align*}
which implies $\lim_{t\rightarrow\infty}\ma{R}(t)\ma{P}=\ma{P}$ since $\ma{P}+\ma{Q}=\ma{I}$. The matrix $\ma{P}$ is an almost global attractor due to $\{\ma{H}\in\R^{n\times n}\,|\,\ma{H}=\ma{R}\ma{P},\,\ma{R}\in\NO\}$ being a set of zero measure in $\mathcal{H}$.
\end{proof}

\section{The Full Attitude}

%Consider the stability of $\ma{R}=\ma{I}$ as an equilibrium of system \eqref{eq:closed2014} on $\SO$. %The system can be described by the evolution of $\ma{H}=\ma{R}\ma{P}$ given by \eqref{eq:RP} and
%%
%\begin{align}
%\md{R}\ma{Q}&=-(\ma{R}\ma{P})\ma{R}\ma{Q}+k(\ma{R}\ma{Q})(\ma{R}\ma{Q})\mtr\ma{Q}-k(\ma{R}\ma{Q})^2, \label{eq:RQ}
%\end{align}
%%
%which is obtained by post-multiplying $\md{R}$, \ie \eqref{eq:closed2014}, by $\ma{Q}$. 
\noindent To establish almost global asymptotic stability  poses a challenge since the set of undesired equilibria is spread out through $\SO$. The proof consists of four parts: (i) LaSalle's invariance principle is used to characterize the set of all equilibria $\mathcal{M}$; (ii) the local stability properties of each equilibria $\ma{R}\in\mathcal{M}$ is studied using the indirect method of Lyapunov; (iii) by normal hyperbolicity of $\mathcal{M}$, it is shown that the $\omega$-limit set of any system trajectory is a singleton; and finally, by (ii) and (iii) it becomes possible to draw conclusions regarding the global behavior of the system based on a local analysis of all $\ma{R}\in\mathcal{M}$.

\subsection{LaSalle's Invariance Principle}

\noindent Let us show that $\ma{R}$ converges to an equilibrium set consisting of symmetric rotation matrices.
 %sMoreover, the identity matrix is the only asymptotically stable equilibrium within this set.

%By defining the three nested sets in a suitable way we are able to form an , invariant set $\mathcal{S}$ such that $\mathcal{S}\cap\NO=\emptyset$ and whose measure is arbitrarily small. Let $\mathcal{S}_1=\{\ma{I}\}$. Denote $\mathrm{DU}(n)=\{\ma{D}\in\U\,|\,\ma[ij]{D}=0,\,\forall\,i\neq j\}$. For a given $\varepsilon\in(0,\infty)$, define the set 
%%
%\begin{align*}
%\NOE=\{\ma{R}\in\SO\,|\,\ma{R}=\ma{U}\ma{\Lambda}\ma{U}^*,\,\ma{U}\in\mathrm{U}(n),\,\ma{\Lambda}\in\mathrm{DU}(n),\,\varepsilon>\min_i|\ma[ii]{\Lambda}-(-1)|\},
%%\NOE=\{\ma{R}\in\SO\,|\,\exists\,\ma{Q}&\in\NO,\,\ma{Q}=\ma{U}\ma{\Lambda}\ma{U}^*,\\
%%\ma{U}&\in\mathrm{U}(n),\,\ma{\Lambda}\in\mathrm{D}(n)\cap\mathrm{U}(n)\\
%%\ma{R}&=\ma{U}(\ma{\Lambda}+\ma{\Delta})\ma{U}^*,\\
%%\Delta&\in\mathrm{D}(n),\, \|\ma{\Delta}\|<\varepsilon\},
%\end{align*}
%%
%corresponding to all elements of $\NO$ and the matrices belonging to an open set around each such element. Take $\mathcal{S}_{\varepsilon}=\SO\backslash\NOE$ to be the state space, where $\varepsilon$ must be small enough that $\ma{I}\in\mathcal{S}_{\varepsilon}$. Let  $\mathcal{S}_2=\{\ma{R}\in\mathcal{S}_{\varepsilon}\,|\,\ma{R}=\ma{P}+\ma{Q}\ma{R}\ma{Q}\}$. Clearly  $\mathcal{S}_1\subset\mathcal{S}_2\subset\mathcal{S}_{\varepsilon}\subset\mathcal{S}$, where $\mathcal{S}=\SO\backslash\NO$.
%
%Say something about needing to prove the invariance of the state space. This may be more of a problem for dynamic control laws.

\begin{proposition}\label{th:symmetric}
The closed-loop dynamics generated by Algorithm \ref{algo:2014} converges to the set of equilibria $\mathcal{M}\subset\SO$ given by
\begin{align*} \mathcal{M}=\{\ma{R}\in\SO\,|\,\ma{R}\mtr=\ma{R},\,[\ma{R},\ma{P}]=\ma{0}\}\subset\{\ma{I}\}\cup\NO.
\end{align*}
\end{proposition}

\begin{proof}
Consider the candidate Lyapunov function $V=\trace(\ma{I}-\ma{R})$ whose time-derivative satisfies
\begin{align}
\dot{V}&=-\trace\ma{P}+\trace\ma{P}\ma{R}^2-k\trace\ma{R}\ma{Q}\mat{R}\ma{Q}+k\trace\ma{R}\ma{Q}\ma{R}\ma{Q}\nonumber\\
&=-\trace\ma{P}^2+\langle\ma{P},\ma{R}^2\rangle-k\|\ma{Q}\ma{R}\ma{Q}\|^2_F+k\trace(\ma{Q}\ma{R}\ma{Q})^2\nonumber\\
&=-\|\ma{P}\|^2_F+\langle\ma{P},\ma{R}^2\rangle-k\|\ma{Q}\ma{R}\ma{Q}\|^2_F+k\langle\ma{Q}\mat{R}\ma{Q},\ma{Q}\ma{R}\ma{Q}\rangle.\label{eq:Vd}
\end{align}
The inequality $\langle\ma{A},\ma{B}\rangle\leq\|\ma{A}\|_*\|\ma{B}\|_2$, where $\ma{A},\,\ma{B}\in\R^{n\times n}$ and $\|\cdot\|_*$ denotes the nuclear norm, implies that $\langle\ma{P},\ma{R}^2\rangle\leq\|\ma{P}\|_*\|\ma{R}^2\|_2=\|\ma{P}\|_*=\rank\ma{P}=\|\ma{P}\|_F^2$. The Cauchy-Schwarz inequality gives $\langle\ma{Q}\mat{R}\ma{Q},\ma{Q}\ma{R}\ma{Q}\rangle\leq\|\ma{Q}\ma{R}\ma{Q}\|^2_F$, \ie $\dot{V}$ is negative semidefinite.
%Since $\ma{P}$ is positive semi-definite, it has a spectral factorization $\ma{P}=\ma{O}\ma{D}\ma{O}\mtr$, where $\ma{O}\in\mathsf{O}(n)$. Moreover, $\ma{Q}=\ma{O}(\ma{I}-\ma{D})\ma{O}\mtr$. It simplifies notion to assign $\ma{D}$ a block structure, 
%%
%\begin{align*}
%\ma{D}=\begin{bmatrix}
%\ma{I} & \ma{0}\\
%\ma{0} & \ma{0}
%\end{bmatrix}.
%\end{align*}	
%%	
%Denote $\ma{M}=\ma{O}\mtr\ma{R}\ma{O}\in\SO$ and $r=\trace\ma{D}=\rank\ma{P}$. Note that
%%
%\begin{align*} \dot{V}&=-r+\trace\ma{D}\ma{O}\mtr\ma{R}^2\ma{O}\ma{D}-\|\ma{O}\ma{D}\mat{O}\ma{R}\ma{O}\ma{D}\mat{O}\|^2+\trace(\ma{O}\ma{D}\mat{O}\ma{R}\ma{O}\ma{D}\mat{O})^2\\
%&=-r+\trace(\ma{D}\ma{M}^2\ma{D})-\|\ma{D}\ma{M}\ma{D}\|^2+\trace(\ma{D}\ma{M}\ma{D})^2\\
%&=-\sum_{i=1}^r\left(1-\sum_{j=1}^n\ma[ij]{M}\ma[ji]{M}\right)-\sum_{i=1}^r\sum_{j=1}^r\ma[ij]{M}^2-\ma[ij]{M}\ma[ji]{M},
%\end{align*}
%%
%which is negative by the Cauchy-Schwarz inequality. 

%If $\ma[0]{R}\in\NO$, then either $\ma{R}$ remains in $\NO$ or it leaves $\NO$ in which case Proposition \ref{prop:RPstable} applies.

The matrix $\ma{R}$  converges to the largest invariant set satisfying $\langle\ma{P},\ma{R}^2\rangle=\|\ma{P}\|_F^2$ and $\langle\ma{Q}\mat{R}\ma{Q},\ma{Q}\ma{R}\ma{Q}\rangle=\|\ma{Q}\ma{R}\ma{Q}\|^2_F$ by LaSalle's invariance principle. The latter equality gives $\ma{Q}\ma{R}\ma{Q}=\ma{Q}\mat{R}\ma{Q}$ whereas the former yields $\ma{R}^2\ma{P}=\ma{P}$ after some calculations. To that end, let $\ma{P}=\ma{O}\ma{\Pi}\mat{O}$, where $\ma{O}\in\mathsf{O}(n)$, express the spectral decomposition of $\ma{P}$. Then $\trace\ma{P}\ma{R}^2=\trace\ma{P}$ implies $\trace\mat{O}\ma{R}^2\ma{O}\ma{\Pi}=\trace\ma{\Pi}$ which requires $\mat{O}\ma{R}^2\ma{O}\ma{\Pi}=\ma{\Pi}$, \ie $\ma{R}^2\ma{P}=\ma{P}$, due to $\ma{\Pi}$ being diagonal and $\mat{O}\ma{R}^2\ma{O}\in\SO$.

Let us show that the closed loop system converges to a set of equilibria. Proposition \ref{prop:RPstable} implies that $(\ma{R}\ma{P})^2=\ma{P}$ asymptotically, which implies $\ma{P}\ma{R}\ma{P}=\mat{R}\ma{P}$. Substitute $\ma{P}=\ma{R}^2\ma{P}$ to obtain $\ma{P}\ma{R}\ma{P}=\mat{R}\ma{P}=\ma{R}\ma{P}$, whereby
\begin{align*}
\ma{R}\ma{P}\ma{R}=\ma{P}\ma{R}\ma{P}\ma{R}=(\mat{R}\ma{P}\mat{R}\ma{P})\mtr=(\ma{R}\ma{P}\ma{R}\ma{P})\mtr=\ma{P}.
\end{align*}
From $\ma{R}\ma{P}\ma{R}=\ma{P}$ and $\ma{Q}\ma{R}\ma{Q}=\ma{Q}\mat{R}\ma{Q}$ it follows that $\md{R}$ given by \eqref{eq:closed2014} is the zero matrix, \ie the system converges to a set of equilibria.

%Note that $\ma{Q}\ma{R}\ma{Q}=\ma{Q}\mat{R}\ma{Q}$ implies that the second difference in \eqref{eq:closed2014} is zero. Moreover, since
%%
%\begin{align*}
%\mat{O}(\ma{P}-\ma{R}\ma{P}\ma{R})\ma{O}=\ma{\Pi}-\mat{O}\ma{R}\ma{O}\ma{\Pi}\mat{O}\ma{R}\ma{O}
%\end{align*}
%% 

Take any eigenpair $(\lambda,\ve{v})$ of $\ma{R}$. Since $\ma{RPR}=\ma{P}$ gives $\ma{PR}=\mat{R}\ma{P}$, it follows that
\begin{align*}
\lambda\|\ma{P}\ve{v}\|^2_2&=\lambda\langle\ma{P}\ve{v},\ma{P}\ve{v}\rangle=\langle\lambda\ma{P}\ve{v},\ve{v}\rangle=\langle\ma{P}\ma{R}\ve{v},\ve{v}\rangle=\langle\mat{R}\ma{P}\ve{v},\ve{v}\rangle\\
&=\langle\ve{v},\ma{P}\ma{R}\ve{v}\rangle=\langle\ve{v},\lambda\ma{P}\ve{v}\rangle=\lambda^*\langle\ve{v},\ma{P}\ve{v}\rangle=\lambda^*\|\ma{P}\ve{v}\|^2_2,
\end{align*}
it either holds that $\lambda^*=\lambda$ or $\ve{P}\ve{v}=\ma{0}$. Consider the latter case. Then $\ma{Q}\ve{v}=\ve{v}$ whereby $\lambda\ve{v}=\ma{Q}\ma{R}\ma{Q}\ve{v}=\ma{Q}\mat{R}\ma{Q}\ve{v}=\lambda^*\ve{v}$. So $\lambda^*=\lambda$, whereby $\lambda\in\{-1,1\}$. Let $\ma{R}=\ma{U}\ma{\Lambda}\ma{U}^*$ denote a spectral factorization of $\ma{R}$. Clearly $\mat{R}=\ma{R}^*=\ma{R}$, \ie $\ma{R}$ is symmetric. Moreover, $\ma{PR}=\mat{R}\ma{P}=\ma{RP}$ implies that $[\ma{R},\ma{P}]=\ma{0}$.\end{proof} %To prove the inclusion, note that if $\ma{R}$ is symmetric, then $\sigma(\ma{R})\subset\R\cap\mathcal{S}^0(\C)=\{-1,1\}$, so either $\ma{R}\in\mathcal{N}$ or $\ma{R}=\ma{I}$.\end{proof}

\begin{remark}
It can be shown that $\|\ma{U}\|^2_F=-2\dot{V}$, where $\dot{V}$ is given by \eqref{eq:Vd}, which bounds the $L^2$-norm of $\ma{U}$ as $\int_0^\infty\|\ma{U}\|^2_F\,\diff t\leq2V(0)$.
\end{remark}

\begin{proposition}\label{prop:cool}
The equilibrium set $\mathcal{M}$ in Proposition \ref{th:symmetric} admits a decomposition as a finite union
\begin{align*}
\mathcal{M}&=\left(\bigcup_{i\in\mathcal{E}(n)}\mathcal{M}_i\cap\mathcal{P}\right)\cup\{\ma{I}\},%^{\left\lfloor\frac{n}{2}\right\rfloor}
\end{align*}
where
\begin{align*} 
\mathcal{M}_i&=\{\ma{R}\in\SO\,|\,\mat{R}=\ma{R},\,V(\ma{R})=2i\},\\
\mathcal{P}&=\{\ma{A}\in\R^{n\times n}\,|\,[\ma{A},\ma{P}]=\ma{0}\},
\end{align*}
$\mathcal{E}(n)=\{2,\ldots,2\left\lfloor\tfrac{n}{2}\right\rfloor\}$, and $V=\trace(\ma{I}-\ma{R})$. Each differentiable manifold $\mathcal{M}_i\cap\mathcal{P}$ is path connected and separated by a continuous function from the others. The $\omega$-limit set of any solution $\ma{R}(t)$ to \eqref{eq:closed2014} either equals $\{\ma{I}\}$ or is a subset of $\mathcal{M}_i\cap\mathcal{P}$ for some $i\in\mathcal{E}$.
\end{proposition}

\begin{proof} The elements of the set $\mathcal{M}$ are symmetric by Proposition \ref{th:symmetric}. Let $-1\in\sigma(\ma{R})$ have algebraic multiplicity $i\in\mathcal{E}$. Then $\trace\ma{R}=(n-i)\cdot1+i\cdot(-1)=n-2i$, \ie $V(\ma{R})=2i$. 
	
Let $\ma{X},\ma{Y}\in\mathcal{M}_i$. Form a curve $\ma{Z}:[0,1]\rightarrow\SO$ from $\ma{X}$ to $\ma{Y}$ by $\ma{Z}(t)=\ma{Y}^{\frac{t}{2}}\ma{X}^{1-t}\ma{Y}^\frac{t}{2}$. Since $-1\in\sigma(\ma{R})$, we need to use a non-principal matrix logarithm $\log$ to calculate $\ma{Z}$. Let the branch-cut of $\log$ be non-real. We may write $\ma{Z}(t)=\exp[\tfrac{t}{2}\log\ma{Y}]\exp[(1-t)\log\ma{X}]\exp[\tfrac{t}{2}\log\ma{Y}]$, since $\ma{X}$ and $\ma{Y}$ have real non-principal matrix logarithms \citep{culver1966existence}. Note that $\ma{Z}$ is a symmetric rotation matrix;  each factor is a matrix function of a symmetric matrix and hence symmetric \citep{higham2008functions},
\begin{align*}	
\mat{Z}(t)={}&\exp[\tfrac{t}{2}\log\ma{Y}]\mtr\exp[(1-t)\log\ma{X}]\mtr\exp[\tfrac{t}{2}\log\ma{Y}]\mtr\\
={}&\exp[\tfrac{t}{2}\log\ma{Y}]\exp[(1-t)\log\ma{X}]\exp[\tfrac{t}{2}\log\ma{Y}]=\ma{Z}(t),\\
\mat{Z}(t)\ma{Z}(t)={}&\exp[-\tfrac{t}{2}\log\ma{Y}]\exp[-(1-t)\log\ma{X}]\exp[-\tfrac{t}{2}\log\ma{Y}]\cdot\\
&\exp[\tfrac{t}{2}\log\ma{Y}]\exp[(1-t)\log\ma{X}]\exp[\tfrac{t}{2}\log\ma{Y}]=\ma{I},\\
\det\ma{Z}(t)={}&\e^{\trace t\log \ma{Y}}\e^{\trace(1-t)\log\ma{X}}=1,
\end{align*}
for all $t\in[0,1]$. Since $\ma{Z}\in\SO$ is symmetric, it follows that $\sigma(\ma{Z})\subset\{-1,1\}$. The trace function is continuous but only assumes integer values on the set of symmetric rotation matrices, \ie $\trace\ma[i]{Z}(t)=2i$ for all $t\in[0,1]$ implying that $\ma{Z}:[0,1]\rightarrow\mathcal{M}_i$. Moreover, $[\ma{P},\ma{Z}]=\ma{0}$ by \citep{higham2008functions} wherefore $\ma{Z}:[0,1]\rightarrow\mathcal{M}_i\cap\mathcal{P}$, thereby establishing that $\mathcal{M}_i\cap\mathcal{P}$ is path connected.

Let $\Omega$ denote the $\omega$-limit set of a solution $\ma{R}$ of \eqref{eq:closed2014} and suppose $\Omega\cap\mathcal{M}_i\neq\emptyset$ for some $i\in\mathcal{E}$. Recall that $V=\trace(\ma{I}-\ma{R})$ decreases in time, as is clear by \eqref{eq:Vd} and that $V|_{\mathcal{M}_j}=2j$ for any $j\in\mathcal{E}$. Since $\ma{R}$ is separated by $V$ from $\mathcal{M}_j$ and $\Omega\cap\mathcal{M}_i\neq\emptyset$, there exists some finite time at which $\ma{R}$ is close enough to $V$ that it cannot come arbitrarily close to $\mathcal{M}_j$ for any $j\in\mathcal{E}$ such that $j>i$ at a later time without violating the decreasingness of $V$, \ie $\Omega\cap\mathcal{M}_j=\emptyset$. Likewise,  $\Omega\cap\mathcal{M}_k=\emptyset$ for all $k\in\mathcal{E}$ such that $k<i$ or else $\Omega\cap\mathcal{M}_i=\emptyset$ by repetition of the same reasoning with $i$ replaced by  $k$. It follows that $\Omega\subset\mathcal{M}_i$ or $\Omega=\{\ma{I}\}$ by Proposition \ref{th:symmetric} and \ref{prop:cool}.\end{proof}

LaSalle's invariance principle is used in Proposition \ref{th:symmetric} to establish convergence to a set $\mathcal{M}$ of equilibria. It remains to determine the region of attraction $\mathcal{R}$ of the identity matrix $\ma{I}\in\mathcal{M}$. Since $V=\trace(\ma{I}-\ma{R})$ decreases in time by \eqref{eq:Vd} and achieves its minimum at $\ma{I}$, it is clear that   $\mathcal{S}=\{\ma{R}\in\SO\,|\,V(\ma{R})<V|_{\mathcal{M}_2}=4\}\subset\mathcal{R}$ due to $\mathcal{S}\cap\mathcal{M}=\{\ma{I}\}$. On $\SOT$, $\mathcal{M}=\mathcal{M}_2\cup\{\ma{I}\}$ wherefore $\mathcal{R}=\mathcal{S}=\SOT\backslash\mathcal{M}_2$. The trace function achieves its global minimum over $\SOT$ on $\mathcal{M}_2$. The general case of $\SO$ is less straightforward since $\mathcal{M}$ contains many saddle points of the trace function, as is illustrated by Example \ref{ex:so4}.

%Determining $\mathcal{R}$ in the general case of $\SO$ is more demanding. 

%To prove almost global stability of the identity matrix on $\SOT$, one can argue that the trace of $\ma{R}$ achieves its global minimum when $-1\in\sigma(\ma{R})$ whereas $\ma{I}$ is the argument of the global maximum. A function cannot converge to a minimum while its derivative is positive, and the only remaining possibility is that $\lim_{t\rightarrow\infty}\ma{R}(t)=\ma{I}$. The case of $\SO$ is less straightforward since $\Omega$ contains many saddle points of the trace function. %This is illustrated by Example \ref{ex:so4}.

\begin{example}\label{ex:so4} Any $\ma{R}\in\SOT$ with $-1\in\sigma(\ma{R})$ is a global minimizer of the trace function, a fact that can be used for stability analysis. By contrast, consider a sequence $\{\ma[i]{R}\}_{i=1}^\infty\subset\mathsf{SO}(4)$ where 
	\begin{align*}	
	\sigma(\ma[i]{R})=\{\exp(i\vartheta),\exp(-i\vartheta),\exp(i\varphi),\exp(-i\varphi)\}.
	\end{align*}
	The sequence of spectra $\{\sigma(\ma[i]{R})\}_{i=1}^\infty$ obtained by setting $\vartheta_i=\frac{1}{n}$, $\varphi_i=\pi-\frac{1}{n+1}$ converges to $\{1,-1\}$ as $n$ goes to infinity with 
	\begin{align*}
	\trace(\ma[i]{R})=2(\cos\tfrac{1}{n}-\cos\tfrac{1}{n+1}),
	\end{align*}
	which approaches zero from below. It follows that $\lim_{i\rightarrow\infty}\ma[i]{R}$ is not even a local minimizer of the trace function.
\end{example}

\subsection{The Indirect Method of Lyapunov}

\noindent A first step towards characterizing the global stability properties of the closed-loop system \eqref{eq:closed2014} is to study local stability by linearizing the dynamics on $\mathcal{M}\subset\SO$. The indirect method of Lyapunov can then be used to determine stability and instability.

\begin{proposition}\label{prop:linearization}
The linearization on $\SO$ of the closed-loop system given by \eqref{eq:closed2014} at an equilibrium $\ma{R}\in\mathcal{M}$  is
\begin{align}
\md{X}=-\ma{XPR}-\ma{RPX}+k\,\ma{RQ}(\mat{X}-\ma{X})\ma{Q},\label{eq:Xd}
\end{align}
where $\ma{X}=\ma{S}\ma{R}$ for some $\ma{S}\in\so$.
\end{proposition}

\begin{proof}
%A perturbation technique is used to linearize rigid-body dynamics, and, more generally, systems that evolve on manifolds \citep{chaturvedi2011rigid}. I

Consider a smooth perturbation of a solution $\ma{R}(t)$ given by $\ma{R}(t,\varepsilon,\ma{S}(t))$, where $\varepsilon\in[0,\infty)$ and $\ma{S}:[0,\infty)\rightarrow\so$.  The perturbed solution is required to be a smooth function $\ma{R}(t,\varepsilon,\ma{S})=\exp(\varepsilon\,\ma{S}(t))\ma{R}(t)\in\SO$ that satisfies \eqref{eq:closed2014} with  $\ma{R}(t,0,\ma{S})=\ma{R}(t)$. Then $\ma{X}(t)=\left.\frac{\diff}{\diff \varepsilon}\ma{R}(t,\varepsilon,\ma{S}(t))\right|_{\varepsilon=0}=\ma{S}(t)\ma{R}(t)$ represents the part of the perturbed solution that is linear in $\varepsilon$. The linearizion on $\SO$ at $\ma{R}\in\mathcal{M}$ is given by
\begin{align*}
\md{X}={}&\left.\tfrac{\diff^2}{\diff t\diff \varepsilon}\ma{R}(t,\varepsilon,\ma{S})\right|_{\varepsilon=0}=\left.\tfrac{\diff^2}{\diff \varepsilon\diff t}\ma{R}(t,\varepsilon,\ma{S})\right|_{\varepsilon=0}=\left.\tfrac{\diff}{\diff \varepsilon}\md{R}(t,\varepsilon,\ma{S})\right|_{\varepsilon=0}\nonumber\\
={}&\tfrac{\diff}{\diff\varepsilon}\{\ma{P}-\exp(\varepsilon\ma{S})\ma{R}\ma{P}\exp(\varepsilon\ma{S})\ma{R}+\nonumber\\
&k\,\exp(\varepsilon\ma{S})\ma{R}\ma{Q}[\mat{R}\exp(\varepsilon\mat{S})-\exp(\varepsilon\ma{S})\ma{R}]\ma{Q}\}|_{\varepsilon=0}\nonumber\\
={}&-\ma{S}\ma{R}\ma{P}\ma{R}-\ma{R}\ma{P}\ma{S}\ma{R}+k\,\ma{S}\ma{R}\ma{Q}(\mat{R}-\ma{R})\ma{Q}+\nonumber\\
&k\,\ma{R}\ma{Q}(\mat{R}\mat{S}-\ma{R})\ma{Q}+k\,\ma{R}\ma{Q}(\mat{R}-\ma{S}\ma{R})\ma{Q}\nonumber\\
={}&-\ma{X}\ma{P}\ma{R}-\ma{R}\ma{P}\ma{X}+k\,\ma{R}\ma{Q}(\mat{X}-\ma{X})\ma{Q}.\qedhere
\end{align*}
\end{proof}

Proposition \ref{prop:zeros} is used in Section \ref{sec:conv} to establish that, roughly speaking, $\mathcal{M}_i\cap\mathcal{P}$ is a normally hyperbolic invariant manifold \citep{aulbach1984continuous}. This is a generalization of the notion of a hyperbolic equilibrium point to the case of equilibrium manifolds. Much like in the case of a single hyperbolic equilibrium, the theory of normally hyperbolic invariant manifolds allows us to conclude that the system \eqref{eq:closed2014} is point-wise convergent, \ie that the $\omega$-limit set of each trajectory is a singleton.

%Essentially, it guarantees that the flow close to the manifold admits an invariant splitting into a stable and unstable part, and that this division persists under small perturbations.

\begin{proposition}\label{prop:zeros}
Let 
\begin{align}
\ma{F}=-\ma{XPR}-\ma{RPX}+k\,\ma{RQ}(\mat{X}-\ma{X})\ma{Q} \label{eq:F}
\end{align}
denote the right-hand side of the linearization of \eqref{eq:closed2014} at $\ma{R}\in\mathcal{M}_i$ given by Proposition \ref{prop:linearization}. The only pure imaginary eigenvalue of $\ma{F}$ is zero. The eigenspace of zero is $\mathcal{X}_i\cap\mathcal{P}$, where
\begin{align*}
\mathcal{X}_i=\{\ma{S}\ma{R}\,|\,\ma{S}\in\so,\,\{\ma{S},\ma{R}\}=\ma{0}\},\quad\mathcal{P}=\{\ma{A}\in\R^{n\times n}\,|\,[\ma{P},\ma{A}]=\ma{0}\}.
\end{align*}
\end{proposition}

\begin{proof} The eigenpairs of the linearization are $(\lambda,\ma{X})\in\C\times\C^{n\times n}$ that satisfy $\lambda\ma{X}=\ma{F}(\ma{X})$, $\ma{X}=\ma{S}\ma{R}$. Consider the case of a purely imaginary eigenvalue, \ie $\lambda=ib$ for some $b\in\R$. Then
\begin{align*}
ib\ma{SR}=-\ma{SR}\ma{P}\ma{R}-\ma{R}\ma{P}\ma{SR}-k\,\ma{R}\ma{Q}(\ma{RS}+\ma{SR})\ma{Q}
\end{align*}
or
\begin{align}
ib\ma{S}=-\ma{SR}\ma{P}-\ma{P}\ma{R}\ma{S}-k\,\ma{Q}(\ma{SR}+\ma{RS})\ma{Q}\label{eq:ibS}
\end{align}
since $[\ma{P},\ma{R}]=\ma{0}$ which is equivalent to  $[\ma{Q},\ma{R}]=\ma{0}$. 

This implies $ib\ma{PSP}=-(\ma{PSPR}+\ma{RPSP})$, $ib\ma{QSQ}=-k(\ma{QSQR}+\ma{RQSQ})$. Denote $\ma{Y}=\ma{PSP}$, assume $b\neq0$, and substitute
\begin{align*}
\ma{Y}=\tfrac{i}{b}(\ma{Y}\ma{R}+\ma{R}\ma{Y})
\end{align*}
into itself to obtain $\ma{Y}=-\tfrac{2}{b^2}(\ma{Y}+\ma{RYR})$. But then $(1+\tfrac{2}{b^2})\|\ma{Y}\|_F=\tfrac{2}{b^2}\|\ma{Y}\|_F$, implying that $\ma{Y}=\ma{PSP}=\ma{0}$. Likewise, it can be shown that $\ma{QSQ}=\ma{0}$. 

Multiply by $\ma{Q}$ from the left and $\ma{P}$ from the right to find that $ib\ma{QSP}=-\ma{QSPR}$. Note that $\ma{R}$ satisfies the requirements for the existence of a square root $\ma{R}^{\frac12}$ \citep{higham2008functions}. Since $\ma{R}^2=\ma{I}$, it holds that $\ma{R}^{\frac32}=\ma{R}^{\frac12}$, which implies $-ib\ma{QSPR}^{\frac12}=\ma{QSPR}^{\frac12}$. Since $b\in\R$, $\ma{QSP}=\ma{0}$, \ie $\ma{SP}=\ma{PSP}$. By analogous reasoning we find $\ma{PSQ}=\ma{0}$, \ie $\ma{PS}=\ma{PSP}$ whereby $[\ma{P},\ma{S}]=\ma{0}$. 

It follows that $\ma{S}=\ma{PSP}+\ma{PSQ}+\ma{QSP}+\ma{QSQ}=\ma{0}$, contradicting that $(ib,\ma{X})$ is an eigenpair of $\ma{F}$. It follows that $b=0$. From 
\begin{align*}
\ma{SR}\ma{P}+\ma{P}\ma{R}\ma{S}+k\,\ma{Q}(\ma{SR}+\ma{RS})\ma{Q}=\ma{0},
\end{align*}
we find $\ma{P}(\ma{SR}+\ma{RS})\ma{P}=\ma{0}$, $\ma{Q}(\ma{SR}+\ma{RS})\ma{Q}=\ma{0}$, $\ma{QSRP}=\ma{0}$, $\ma{PRSQ}=\ma{0}$. The two last equalities also yield $\ma{PSRQ}=\ma{0}$, $\ma{QRSP}=\ma{0}$ by $[\ma{P},\ma{R}]=\ma{0}$, $[\ma{P},\ma{S}]=\ma{0}$. Altogether, $\ma{SR}+\ma{RS}=\ma{0}$ or $\{\ma{S},\ma{R}\}=\ma{0}$.\end{proof}

%Note that $\mat{X}=-\ma{R}\ma{S}=\ma{S}\ma{R}=\ma{X}$ for  $\ma{S}\in\mathcal{S}_i$. Then
%%
%\begin{align*}
%\md{X}=-\ma{X}\ma{P}\ma{R}-\ma{R}\ma{P}\ma{X}=-\ma{S}\ma{R}\ma{P}\ma{R}-\ma{R}\ma{P}\ma{S}\ma{R}=-\ma{S}\ma{P}+\ma{P}\ma{S}=\ma{0}
%\end{align*}
%%
%if and only if $[\ma{P},\ma{S}]=\ma{0}$, \ie the assumption of $\ma{S}\in\mathcal{P}$ is necessary. Also note that since $\ma{X}$ is symmetric, it has an orthonormal basis of eigenvectors. In particular, the geometric and algebraic multiplicity of all eigenvalues coincide, so the number of eigenvalues with real part zero equals $\dim\ker\ma{F}$ and the number with nonzero real part equals $\dim\mathrm{im}\,\ma{F}=\tfrac12n(n-1)-\dim\ker\ma{F}$.\end{proof}

\begin{proposition}\label{prop:indirect}
The linearized system given by Proposition \ref{prop:linearization} is exponentially unstable at all $\ma{R}\in\mathcal{M}_i\backslash\{\ma{I}\}$ for all $i\in\mathcal{E}$. The number of eigenvalues of $\ma{F}$ with nonzero real part is 
\begin{align*}
\dim\mathrm{im}\,\ma{F}=\dim\SO-\dim\ker\,\ma{F}.
\end{align*}
The linearization is exponentially stable at $\ma{R}=\ma{I}$.
\end{proposition}

\begin{proof} Since $[\ma{P},\ma{R}]=\ma{0}$, $\ma{P}$ and $\ma{R}$ are simultaneously diagonalizable, \ie they share an orthonormal basis of eigenvectors. Suppose that there are two linearly independent eigenpairs, $(-1,\ve{v})$ and $(-1,\ve{u})$ of $\ma{R}$ such that $(1,\ve{v})$ and $(1,\ve{u})$ are eigenpairs of either $\ma{P}$ or $\ma{Q}$. Set $\ma{S}=\ve{u}\otimes\ve{v}-\ve{v}\otimes\ve{u}$, then either
\begin{align*}
\ma{F}&=-\ma{SR}\ma{P}-\ma{P}\ma{R}\ma{S}-k\,\ma{Q}(\ma{SR}+\ma{RS})\ma{Q}=2\ma{S},
\end{align*}
or
\begin{align*}
\ma{F}&=-\ma{SR}\ma{P}-\ma{P}\ma{R}\ma{S}-k\,\ma{Q}(\ma{SR}+\ma{RS})\ma{Q}=2k\ma{S},
\end{align*}
\ie either $(2,\ma{S}\ma{R})$ is an eigenpair of $\ma{F}$ or $(2k,\ma{S}\ma{R})$ is.

Suppose the above construction is impossible. Since the eigenvalue multiplicity $m(-1)=i$ is even, it must be the case that $m(-1)=2$ and $(-1,\ve{u})$, $(-1,\ve{v})$ are eigenpairs of $\ma{R}$ such that $(1,\ve{u})$ is an eigenpair of $\ma{P}$ and $(1,\ve{v})$ is an eigenpair of $\ma{Q}$. Set $\ma{S}=\ve{u}\otimes\ve{v}-\ve{v}\otimes\ve{u}\in\so$ whereby
\begin{align*}
\ma{F}&=-\ma{SR}\ma{P}-\ma{P}\ma{R}\ma{S}-k\,\ma{Q}(\ma{SR}+\ma{RS})\ma{Q}\\
&=-(\ve{u}\otimes\ve{v}-\ve{v}\otimes\ve{u})\ma{RP}-\ma{PR}(\ve{u}\otimes\ve{v}-\ve{v}\otimes\ve{u})\\
&=-\ve{v}\otimes\ve{u}+\ve{u}\otimes\ve{v}=\ma{S},
\end{align*}
\ie $(1,\ma{SR})$ is an eigenpair of $\ma{F}$.

Consider the construction of eigenvectors of $\ma{F}$ for $\ma{R}\in\mathcal{M}_i$. Let there be $m$ eigenpairs $(\lambda_j,\ve[j]{v})$ of $\ma{R}$ such that $(1,\ve[j]{v})$ is an eigenpair of $\ma{P}$, $j\in\{1,\ldots,m\}$ and $i-m$ eigenpairs such that $(1,\ve[j]{v})$, $j\in\{m+1,\ldots,i\}$, is an eigenpair of $\ma{Q}$. Set $\ma{S}=\ve[k]{v}\otimes\ve[l]{v}-\ve[l]{v}\otimes\ve[k]{v}$ for some $l,k\in\{1,\ldots,i\}$. Then 
\begin{align*}
\ma{F}={}&-\ma{SR}\ma{P}-\ma{P}\ma{R}\ma{S}-k\,\ma{Q}(\ma{SR}+\ma{RS})\ma{Q}=\\
={}&-\lambda_l\ve[k]{v}\otimes\ve[l]{v}\ma{P}+\lambda_k\ve[l]{v}\otimes\ve[k]{v}\ma{P}-\lambda_k\ma{P}\ve[k]{v}\otimes\ve[l]{v}+\lambda_l\ma{P}\ve[l]{v}\otimes\ve[k]{v}+\\
&-k\,\ma{Q}(\lambda_l\ve[k]{v}\otimes\ve[l]{v}-\lambda_k\ve[l]{v}\otimes\ve[k]{v}+\lambda_k\ve[k]{v}\otimes\ve[l]{v}-\lambda_l\ve[l]{v}\otimes\ve[k]{v})\ma{Q}\\
={}&\lambda_l(\ma{P}\ve[l]{v}\otimes\ve[k]{v}-\ve[k]{v}\otimes\ve[l]{v}\ma{P})+\lambda_k(\ve[l]{v}\otimes\ve[k]{v}\ma{P}-\ma{P}\ve[k]{v}\otimes\ve[l]{v})+\\
&-k\,\ma{Q}[\lambda_l(\ve[k]{v}\otimes\ve[l]{v}-\ve[l]{v}\otimes\ve[k]{v})+\lambda_k(\ve[k]{v}\otimes\ve[l]{v}-\ve[l]{v}\otimes\ve[k]{v})]\ma{Q}.
\end{align*}

There are three cases to consider: either $l,k\in\{1,\ldots,m\}$ whereby $\ma{F}=-(\lambda_l+\lambda_k)\ma{S}$, $l\in\{1,\ldots,m\}$ and $k\in\{m+1,\ldots,n\}$  whereby $\ma{F}=-\lambda_l\ma{S}$, or $l,k\in\{m+1,\ldots,n\}$ whereby $\ma{F}=-k(\lambda_l+\lambda_k)\ma{S}$. Assume there are $j\leq i$ indices $l\in\{1,\ldots,m\}$ such $\lambda_l=-1$. Then there are $i-j$ indices  $l\in\{m+1,\ldots,n\}$ such that $\lambda_l=-1$. Let $r_i$ denote the number of eigenvectors of $\ma{F}$ at $\ma{R}\in\mathcal{M}_i\cap\mathcal{P}$ with nonzero eigenvalue. Then
\begin{align*}
r_i=\binom{j}{2}+\binom{m-j}{2}+\binom{m}{1}\binom{n-m}{1}+\binom{i-j}{2}+\binom{n-m-(i-j)}{2}.
\end{align*}

The combinatorial identity
\begin{align}
\binom{n}{2}=\binom{k}{2}+\binom{k}{1}\binom{n-k}{1}+\binom{n-k}{2},\label{eq:combinatorial}
\end{align}
for any $n\in\N$ and $k\in\{1,\ldots,n\}$, is obtained by noting that to chose two numbers in $\{1,\ldots,n\}$ can be done by either choosing two in $\{1,\ldots,k\}$, one in $\{1,\ldots,k\}$ and one in $\{k+1,\ldots,n\}$, or two in $\{k+1,\ldots,n\}$. Repeated application of \eqref{eq:combinatorial} in the expression for $r_i$ yields
\begin{align*}
r_i&=\binom{n}{2}-\binom{j}{1}\binom{m-j}{1}-\binom{i-j}{1}\binom{n-m}{1}\\ &=\dim\ts[\SO]{\ma{R}}-\dim\ker\ma{F}=\dim\mathrm{im}\,\ma{F}.                   
\end{align*}

Consider the case of $\ma{R}=\ma{I}$. Let $p=\rank\ma{P}$. By reasoning as above, there are $\binom{p}{2}$ linearly independent eigenpairs of $\ma{F}$ on of the form $(-2,\ma{S})$, $\binom{p}{1}\binom{n-p}{1}$ on the form $(-1,\ma{S})$, and $\binom{n-p}{2}$ on the form $(-2k,\ma{S})$. In total, there are
\begin{align*}
\binom{p}{2}+\binom{p}{1}\binom{n-p}{1}+\binom{n-p}{2}&=\binom{n}{2}=\dim\so
\end{align*}
linearly independent eigenvectors with negative eigenvalues. %, which is the maximum possible number of linearly independent eigenvectors with $\so$ as domain. 
It follows that the identity matrix is an exponentially stable equilibrium of the closed-loop dynamics \eqref{eq:closed2014}. \end{proof}

\subsection{Point-Wise Convergence}\label{sec:conv}

\noindent It remains to determine if each solution $\ma{R}$ to \eqref{eq:closed2014} converges to a single equilibrium within its $\omega$-limit set $\Omega\subset\mathcal{M}$ or if the asymptotic behavior of the closed-loop system is more complex than that. This property, so-called point-wise convergence \citep{lageman2007convergence}, allows us to draw conclusions regarding the region of attraction of exponentially unstable equilbria \citep{freeman2013global}. 

\begin{proposition}[B. Aulbach \citep{aulbach1984continuous}]\label{prop:aulbach}
Consider an autonomous system 
\begin{align}
\vd{x}=\ma{f}(\ve{x}),\label{eq:xd}
\end{align}
where $\ve{f}\in\mathcal{C}^3(\R^n,\R^n)$. Suppose \eqref{eq:xd} has a differentiable manifold $\mathcal{M}$ of equilibrium points. Let $\ma{x}(t)$ be any solution of \eqref{eq:xd} with $\omega$-limit set $\Omega$. Suppose  that 
\begin{itemize}
\item[(i)] there exists a point $\ve{y}\in\Omega$, \ie $\Omega$ is nonempty,
\item[(ii)] there exists a neighborhood $\mathcal{B}$ of $\ve{y}$ such that $\Omega\cap \mathcal{B}\subset \mathcal{M}$,
\item[(iii)] $n-\dim\mathcal{M}$ eigenvalues of the Jacobian $\ma{J}(\ve{y})$ of $\ve{f}$ evaluated at $\ve{y}$ have nonzero real parts.
\end{itemize}
Then $\lim_{t\rightarrow\infty}\ve{x}(t)=\ve{y}$, \ie $\Omega=\{\ve{y}\}$.
\end{proposition} 

\begin{proposition}\label{prop:conv}
Any solution $\ma{R}$ of \eqref{eq:closed2014} with initial condition on $\SO$ converges to an equilibrium point.
\end{proposition}

\begin{proof}
The proof is by verification of property (i)--(iii) of Proposition \ref{prop:aulbach} with respect to the system given by \eqref{eq:closed2014} and the differentiable manifold of equilibria $\mathcal{M}_i\cap\mathcal{P}$ characterized by Proposition \ref{prop:cool}.
% where $\mathcal{M}_i\cap\mathcal{P}$ can be taken as the differentiable manifold by Proposition \ref{prop:cool}. %The right-hand side of \eqref{eq:closed2014} is $\mathcal{C}^3$ since it is $\mathcal{C}^\omega$. Consider the manifold $\mathcal{R}_i\cap\mathcal{P}$ which is smooth since?
	
(i) Since $\SO$ is compact, any trajectory can be sampled as $\{\ma{R}(t_j)\}_{j=0}^\infty$ for some $\{t_j\}_{j=0}^\infty\subset[0,\infty)$ such that $\lim_{j\rightarrow\infty}\ma{R}(t_j)$ exists by the Bolzano-Weierstrass theorem. The limit $\lim_{j\rightarrow\infty}\ma{R}(t_j)$ belongs to the $\omega$-limit set $\Omega$.
	
(ii) The inclusion $\Omega\subset\mathcal{M}_i\cap\mathcal{P}$ holds by Proposition \ref{prop:cool}. %For any $\mathcal{B}\subset\SO$ it hence holds that $\Omega\cap\mathcal{B}\subset\mathcal{M}_i\cap\mathcal{P}$.
	
(iii) The number of eigenvalues with nonzero real part is $\dim\mathrm{im}\,\ma{F}$ by Proposition \ref{prop:zeros}. Since $\mathcal{M}_i\cap\mathcal{P}$ is connected it holds that $\dim\mathcal{M}_i\cap\mathcal{P}=\dim\ts[\mathcal{M}_i\cap\mathcal{P}]{\ma{R}}$. Note that  $\trace(\ma{I}-\ma{R})=2i$ for all $\ma{R}\in\mathcal{M}_i$, as required by Proposition \ref{prop:cool}, is implied by $\ts[\mathcal{M}_i]{\ma{R}}\subset\{\ma{A}\in\R^{n\times n}\,|\,\mat{A}=\ma{A}\}$ since the trace function is integer valued over symmetric rotation matrices. Furthermore,
\begin{align}
\ts[\mathcal{M}_i\cap\mathcal{P}]{\ma{R}}&=\ts[\SO]{\ma{R}}\cap\{\ma{A}\in\R^{n\times n}\,|\,[\ma{P},\ma{A}]=\ma{0},\,\mat{A}=\ma{A}\}\nonumber\\
&=\{\ma{X}\in\R^{n\times n}\,|\,\ma{X}=\ma{S}\ma{R},\,\ma{S}\in\so,\,[\ma{P},\ma{X}]=\ma{0},\,\mat{X}=\ma{X}\}\nonumber\\
&=\{\ma{X}\,|\,\ma{X}=\ma{SR},\,\ma{S}\in\so,\,[\ma{P},\ma{S}]=\ma{0},\,\{\ma{R},\ma{S}\}=\ma{0}\}\nonumber\\
&=\mathcal{X}_i\cap\mathcal{P}=\ker\ma{F},\label{eq:ts}
\end{align}
where $\ker\ma{F}=\mathcal{X}_i\cap\mathcal{P}$ is characterized by Proposition \ref{prop:zeros}. 
%and $\dim\ts[\mathcal{M}_i\cap\mathcal{P}]{\ma{R}}=\dim\mathcal{S}\cap\mathcal{P}$ due to \eqref{eq:ts} and the non-singularity of $\ma{R}$. 
Recall the result of Proposition \ref{prop:indirect}. %the rank-nullity theorem for linear maps, $\dim\mathrm{im}\,\ma{F}+\dim\mathrm{ker}\,\ma{F}=\dim\ts[\SO]{\ma{R}}$, where we calculate the image and kernel with  $\ts[\SO]{\ma{R}}$ as the domain of $\ma{F}$. 
 The number of eigenvalues with nonzero real part is
\begin{align*}
\dim\SO-\dim\mathrm{ker}\,\ma{F}&=\dim\SO-\dim\mathcal{M}_i\cap\mathcal{P}.\qedhere
\end{align*}
\end{proof}

\subsection{Global Stability Analysis}

\noindent We are now ready to state and prove one of the two main results of this paper, Theorem \ref{th:stability}. For the proof, Proposition \ref{prop:unstable} is required. It gives conditions under which the local stability by the first approximation of all equilibria can be used to infer global stability properties of the entire system. Our work thus far ensures that the conditions of Proposition \ref{prop:unstable} are fulfilled with respect to the undesired equilibria contained in $\mathcal{M}\backslash\{\ma{I}\}$

\begin{proposition}[R.\um{}A. Freeman \citep{freeman2013global}]\label{prop:unstable}
Consider a system of the form
\begin{align*}
\vd{x}=\ve{f}(\ve{x}),
\end{align*}
where $\ve{f}$ is a vector field on an $n$-dimensional, connected,
smooth Riemannian manifold $\mathcal{X}$. Suppose $\mathcal{S}\subset\mathcal{X}$ is a set of  equilibria, that $\ve{f}$ is $\mathcal{C}^1$, that each equilibrium in $\mathcal{S}$ is exponentially unstable, and suppose that
\begin{align}
\mathcal{R}(\mathcal{S})=\cup_{\ve{x}\in\mathcal{S}}\mathcal{R}(\ve{x}),\label{eq:convergent}
\end{align}
where $\mathcal{R}$ maps a set of equilibria to the union of their regions of attraction. Then $\mathcal{R}(\mathcal{S})$ is of measure zero and meager on $\mathcal{X}$.
\end{proposition}

In passing we note that it is possible for a set to attract trajectories that do not have a limit, \ie condition \ref{eq:convergent} may fail to hold under unfavorable circumstances. An example of such behavior, where an exponentially unstable set is globally attractive, is provided in \citep{freeman2013global}.

\begin{theorem}\label{th:stability}
The identity matrix is an almost globally asymptotically stable equilibrium of the closed-loop dynamics generated by Algorithm \ref{algo:2014}. The rate of convergence is locally exponential. The set of initial conditions from which convergence to the identity matrix fails is meager in $\SO$.
\end{theorem}

\begin{proof}
All trajectories converge to equilibria by Proposition \ref{prop:conv}, which implies that condition \eqref{eq:convergent} of Proposition \ref{prop:unstable} is fulfilled. The set $\mathcal{M}\backslash\{\ma{I}\}$ consists of exponentially unstable equilibria by Proposition \ref{prop:indirect}. By Proposition \ref{prop:unstable}, the region of attraction of $\mathcal{M}\backslash\{\ma{I}\}$ is meager and has zero measure on $\SO$. It follows that the identity matrix is almost globally attractive. The identity matrix is exponentially stable by Proposition \ref{prop:indirect} and by the principle of stability in the first approximation.
\end{proof}

\section{Exact Solutions on $\SOT$}

\label{sec:exact}

\noindent We provide the exact solutions in the case of $\SOT$. This case is the most interesting from an applications point of view. If necessary, exchange the roles of $\ma{P}$ and $\ma{Q}$ as well as the coordinates such that $\ma{P}=\ve[1]{e}\vet[1]{e}$ and $\ma{Q}=\ma{I}-\ve[1]{e}\vet[1]{e}$. This can be done without loss of generality since either $(\rank\ma{P},\rank\ma{Q})=(1,2)$,  $(\rank\ma{P},\rank\ma{Q})=(2,1)$, $(\rank\ma{P},\rank\ma{Q})=(3,0)$, or $(\rank\ma{P},\rank\ma{Q})=(0,3)$. The first two cases imply $\ma{P}=\ve{e}\vet{e}$ and $\ma{Q}=\ve{e}\vet{e}$ respectively for some $\ve{e}\in\St$. A change of coordinates yields $\ve{e}=\ve[1]{e}$. The case of $\ma{Q}=\ve[1]{e}\vet[1]{e}$ is easier because the control \eqref{eq:control2014} simplifies to $\ma{U}=\ma{P}\mat{R}-\ma{R}\ma{P}$. The solution is given implicitly by Proposition \ref{th:RP} and the constraint $\ma{R}\in\SOT$. It is provided explicitly by \citep{markdahl2013analytical,markdahl2015automatica}. The last two cases imply $\ma{P}=\ma{I}$ or $\ma{Q}=\ma{I}$ respectively and is covered in Proposition \ref{th:RP}, as well as in \citep{markdahl2013analytical} since it can be generated by a simpler algorithm. Note that the first case is the geodesic control law for the reduced attitude.

%\subsection{Block-Matrix Dynamics}

Let us denote
\begin{align}
\ma{R}=\begin{bmatrix}
r_{11} & \ve[12]{r}\\
\ve[21]{r} & \ma[22]{R}
\end{bmatrix}.\label{sys:block}
\end{align}
Algorithm \ref{algo:2014} with $\ma{P}=\ve[1]{e}\vet[1]{e}$ and $\ma{Q}$ redefined as $\ma{Q}=(\ma{I}-\ma{P})$ generates the following system on $\SO$: $\dot{r}_{11}=1-r_{11}^2$, $\vd[21]{r}=-r_{11}\ve[21]{r}$, $\vd[12]{r}=\ve[12]{r}[-r_{11}\ma{I}+k(\mat[22]{R}-\ma[22]{R})]$, and $\md[22]{R}=-\ve[21]{r}\ve[12]{r}+k\ma[22]{R}(\mat[22]{R}-\ma[22]{R})$.

Proposition \ref{th:RP} tells us that
\begin{align*}
\ma{R}(t)\ma{P}&=\begin{bmatrix}
r_{11} & \ve{0}\\
\ve[21]{r} & \ma{0}
\end{bmatrix}\\
&=[\sinh(\ma{P}t)+\cosh(\ma{P}t)\ma[0]{R}\ma{P}][\cosh(\ma{P}t)+\sinh(\ma{P}t)\ma[0]{R}\ma{P}]\inv.
\end{align*}
On the form of the block-matrix partition in \eqref{sys:block}, the solutions are given by
\begin{align*}
r_{11}(t)=\tanh(t+\Atanh r_{11,0}), \quad \ve[21]{r}(t)=\frac{\sech t}{1+\tanh(t)r_{11,0}}\ve[21,0]{r},
\end{align*}
for all $r_{11,0}\in(-1,1]$, see \citep{markdahl2013analytical}. %and where $r_{11,0}=r_{11}(0)$, $\ve[11,0]{r}=\ve[11,0]{r}(0)$
It remains to solve \eqref{sys:block} for $\ve[12]{r}(t)$ and $\ma[22]{R}(t)$. Note however that due to the constraints that define $\SOT$ it will suffice to solve the equations defining a subset of the elements of $\ve[12]{r}(t)$ and $\ma[22]{R}(t)$ to determine $\ma{R}(t)$.

%The dynamics \eqref{eq:RQ} of $\ma{R}\ma{Q}$ can be rewritten as a matrix differential Riccati equation. A solution can then be attempted using the adjoint equations technique. However, the resulting linear system is time-dependent and does not satisfy the commutativity relation required to compute the solution using the exponential matrix. Instead, the following approach is taken: the constraints that define $\SOT$ are substituted into System \ref{sys:block} whereby a decoupled function of the states is found. Solving for this function yields sufficient information to calculate the states. 

Use the relation on $\SO$ provided by Lemma \ref{eqB:relations} in Appendix \ref{appB:lemmas} to find that
\begin{align*}
\md[22]{R}&=-r_{11}\ma[22]{R}-\ma{S}\ma[22]{R}\ma{S}+k\trace(\ma[22]{R}\ma{S})\ma[22]{R}\ma{S}.
\end{align*}
Multiplying by $\ma{S}$ yields
\begin{align*}
\md[22]{R}\ma{S}&=-r_{11}\ma[22]{R}\ma{S}+\ma{S}\ma[22]{R}-k\trace(\ma[22]{R}\ma{S})\ma[22]{R}.
\end{align*}
Take the trace and substitute the relations regarding $\SOT$ from Lemma \ref{eqB:relations} to obtain
\begin{align}
\trace\md[22]{R}&=k(1+r_{11})^2+(1-r_{11})\trace\ma[22]{R}-k(\trace\ma[22]{R})^2,\label{eq:trR22}\\
\trace\md[22]{R}\ma{S}&=(1-r_{11}-k\trace\ma[22]{R})\trace\ma[22]{R}\ma{S}.\label{eq:trR22S}
\end{align}

\begin{proposition}\label{propB:tr}
The unique solution to \eqref{eq:trR22} and \eqref{eq:trR22S} for $\ma[0]{R}\in\SO$ is given by  
\begin{align*}
\trace\ma[22]{R}(t)={}&\tanh\{f[\ma[0]{R}]-k\log[1-r_{11}(t)]\}\{1+r_{11}[t]\},\\
\trace\ma[22]{R}(t)\ma{S}={}& g\{\ma[0]{R}\}\exp\{t\}\sech\{t+\Atanh r_{11,0}\}\cdot\\
&\sech\{f[\ma[0]{R}]-k\log[1-r_{11}(t)]\},
\end{align*}
where 
\begin{align*}
f(\ma[0]{R})&=\Atanh[(1+r_{11,0})\trace\ma[22,0]{R}]+k\log[1-r_{11,0}],\\ 
r_{11}(t)&=\tanh(t+\Atanh r_{11,0}),\\ g(\ma[0]{R})&=\{\sech[\Atanh r_{11,0}]\sech[\Atanh(1+r_{11,0})\trace\ma[22,0]{R}]\}\inv\trace\ma[0]{R}\ma{S}.
\end{align*}
\end{proposition}

\begin{proof}
Global existence and uniqueness is implied by Lemma \ref{leB:unique} in Appendix \ref{appB:lemmas}. It remains to verify that the proposed solution solves the required system. Note that 
\begin{align*}
\trace\md[22]{R}(t)={}&(1-\tanh^2\{f[\ma[0]{R}]-k\log[1-r_{11}(t)]\})(-1)^2k\smash{\tfrac{1-r_{11}^2\{t\}}{1-r_{11}\{t\}}}(1+r_{11}\{t\})\\
&+\tanh(f\{\ma[0]{R}\}-\log\{1-r_{11}[t]\})(1-r_{11}^2\{t\})\\
={}&k(1-\tanh^2\{f[\ma[0]{R}]-\log[1-r_{11}(t)]\})(1+r_{11}\{t\})^2
+\\
&\tanh(f\{\ma[0]{R}\}-\log\{1-r_{11}[t]\})(1+r_{11}\{t\})(1-r_{11}\{t\})\\
={}&k[1+r_{11}(t)]^2+[1-r_{11}(t)]\trace\ma[22]{R}(t)-k[\trace\ma[22]{R}(t)]^2,\\
%
%\trace\md[22]{R}(t)\ma{S}={}
&\trace\ma[22]{R}\{t\}\ma{S}-g\{\ma[0]{R}\}\tanh\{t+\Atanh r_{11,0}\}\trace\ma[22]{R}\{t\}\ma{S}-\\
&g\{\ma[0]{R}\}\exp\{t+\Atanh r_{11,0}\}\sech\{t+\Atanh r_{11,0}\}\cdot\\
&\sech\{f[\ma[0]{R}]-k\log[1-r_{11}(t)]\}\tanh\{f[\ma[0]{R}]-\log[1-r_{11}(t)]\}\cdot\\
&\{-1\}^2k\smash{\tfrac{1-r_{11}^2[t]}{1-r_{11}[t]}}\\
={}&[1-r_{11}(t)-k\trace\ma[22]{R}(t)]\trace\ma[22]{R}(t)\ma{S}.
\end{align*}
Moreover, $\trace\ma[22]{R}(0)=\tanh(\Atanh\trace\ma[22,0]{R})=\trace\ma[22,0]{R}$ and  $\trace\ma[22]{R}(0)\ma{S}=\trace\ma[22,0]{R}\ma{S}$.%\{\sech(\Atanh r_{11,0})\sech[\allowbreak\Atanh(1+r_{11,0})\trace\ma[22,0]{R}]\}\inv\trace\ma[0]{R}\ma{S}=\trace\ma[0]{R}\ma{S}$.
\end{proof}

\begin{remark}
The use of the principal inverse hyperbolic tangent $\Atanh:\C\rightarrow\C\cup\{\infty\}$, as described in Section \ref{secB:prel}, is convenient here since $(1+r_{11})\trace\ma[22]{R}\in(-4,4]$ whereas $\Atanh:(-1,1)\rightarrow\R$. The appearance of a discontinuous function in the exact solutions need not lead to a loss of continuous dependence on the initial conditions since the inverse hyperbolic tangent only appears as part of an argument of the hyperbolic tangent. If required, it is possible to find an expression for the exact solutions that does not rely on the use of the inverse hyperbolic tangent by applying a sum of arguments formula. This would however result in expressions that make the proof of Proposition \ref{propB:tr} clunky. See \citep{markdahl2012exact} for more details. 
\end{remark}

\begin{theorem}\label{thB:sol}
The solution to the closed-loop system generated by Algorithm \ref{algo:2014} in the case of $\ma{R}(0)=\ma[0]{R}\notin\NO$ is given by 
\begin{align*}
r_{11}(t)=\smash{\tanh(t+\Atanh r_{11,0})},\quad \ve[21]{r}(t)=\frac{\sech t}{1+\tanh(t)r_{11,0}}\ve[21,0]{r},
\end{align*}
which specify $\ma{R}\ve[1]{e}$, and the unique solution to the following linear system
\begin{align*}
\begin{bmatrix}
\maspace\phantom{-}\,\,(\ma{R}\ve[1]{e})\mtr & \maspace\phantom{-}\ve{0}\\
\maspace\phantom{-}\ma{S}(\ma{R}\ve[1]{e}) & \maspace-\ma{I}\\
\maspace\phantom{-}\vet[2]{e} & \maspace\phantom{-}\,\,\vet[3]{e}\\
\maspace-\vet[3]{e} & \maspace\phantom{-}\,\,\vet[2]{e}
\end{bmatrix}\begin{bmatrix}
\ma{R}\ve[2]{e}\\
\ma{R}\ve[3]{e}
\end{bmatrix}=\begin{bmatrix}
0\\
\ve{0}\\
\trace\ma[22]{R}\\
\trace\ma[22]{R}\ma{S}
\end{bmatrix},
\end{align*}
where $\ma{S}:\R^3\rightarrow\sot$ is the map defined by $\ma{S}(\ve{x})\ve{y}=\ve{x}\times\ve{y}$ for all $\ve{x},\ve{y}\in\R^3$ and $\trace \ma[22]{R}$, $\trace\ma[22]{R}\ma{S}$ are given by Proposition \ref{propB:tr}.
\end{theorem}

\begin{proof}
The first equation follows from $\ma{R}\mat{R}\hspace{-1mm}=\ma{I}$. The second equation states that $\ma{R}\ve[1]{e}\times\ma{R}\ve[2]{e}=\ma{R}\ve[3]{e}$. The third and fourth equation follow from Proposition \ref{propB:tr}. 

Let us verify the uniqueness of the solution under the stated assumptions. Let $\ma{A}$ denote the system matrix. The matrix $\ma{A}$ is nonsingular if and only if 
%
%\vspace*{-0.4cm}
\begin{align}
\ma{B}&=\ma{A}\mtr\ma{A}=\begin{bmatrix}
\ma{R}\ve[1]{e} & -\ma{S}(\ma{R}\ve[1]{e}) & \ve[2]{e} & -\ve[3]{e}\\
\ma{0} & -\ma{I} & \ve[3]{e} & \phantom{-}\ve[2]{e}
\end{bmatrix}\begin{bmatrix}
\maspace\phantom{-}\,\,(\ma{R}\ve[1]{e})\mtr & \maspace\phantom{-}\ve{0}\\
\maspace\phantom{-}\ma{S}(\ma{R}\ve[1]{e}) & \maspace-\ma{I}\\
\maspace\phantom{-}\vet[2]{e} & \maspace\phantom{-}\,\,\vet[3]{e}\\
\maspace-\vet[3]{e} & \maspace\phantom{-}\,\,\vet[2]{e}
\end{bmatrix}\nonumber\\
&=\begin{bmatrix}
\ma{R}\ve[1]{e}\vet[1]{e}\mat{R}-\ma{S}(\ma{R}\ve[1]{e})^2+\ve[2]{e}\vet[2]{e}+\ve[3]{e}\vet[3]{e} & \ma{S}(\ma{R}\ve[1]{e})+\ve[2]{e}\vet[3]{e}-\ve[3]{e}\vet[2]{e}\\
-\ma{S}(\ma{R}\ve[1]{e})+\ve[3]{e}\vet[2]{e}-\ve[2]{e}\vet[3]{e} & \ma{I}+\ve[3]{e}\vet[3]{e}+\ve[2]{e}\vet[2]{e}
\end{bmatrix}\nonumber\\
&=\begin{bmatrix}\ma[11]{B} & \ma[12]{B}\\
\ma[21]{B} & \ma[22]{B}\end{bmatrix}\label{eqB:unique}
\end{align}
is nonsingular. The matrix $\ma{B}$ is nonsingular if one of its diagonal blocks and the Schur complement of that block are both nonsingular. Note that 
$\ma[22]{B}\inv=(\ma{I}+\ve[3]{e}\vet[3]{e}+\ve[2]{e}\vet[2]{e})\inv=\ma{I}-\tfrac12(\ve[3]{e}\vet[3]{e}+\ve[2]{e}\vet[2]{e})$. The Schur complement of $\ma[22]{B}$ is
\begin{align}
\ma{C}={}&\ma[11]{B}-\ma[12]{B}\ma[22]{B}\inv\ma[21]{B}\nonumber\\
={}&\ma{R}\ve[1]{e}\vet[1]{e}\mat{R}+(\ve[2]{e}\vet[3]{e}-\ve[3]{e}\vet[2]{e})\ma[22]{B}\inv(\ve[2]{e}\vet[3]{e}-\ve[3]{e}\vet[2]{e})+\ve[2]{e}\vet[2]{e}+\ve[3]{e}\vet[3]{e}+\nonumber\\
{}&\ma{S}(\ma{R}\ve[1]{e})(\ma[22]{B}\inv-\ma{I})\ma{S}(\ma{R}\ve[1]{e})\nonumber\\
%={}&\ma{R}\ve[1]{e}\vet[1]{e}\mat{R}+\tfrac12(\ve[2]{e}\vet[2]{e}+\ve[3]{e}\vet[3]{e})-\tfrac12\ma{S}(\ma{R}\ve[1]{e})(\ve[2]{e}\vet[2]{e}+\ve[3]{e}\vet[3]{e})\ma{S}(\ma{R}\ve[1]{e})\nonumber\\
={}&\ma{R}\ve[1]{e}\vet[1]{e}\mat{R}+\tfrac12(\ve[2]{e}\vet[2]{e}+\ve[3]{e}\vet[3]{e})+\tfrac12\ma{S}(\ma{R}\ve[1]{e})(\ve[2]{e}\vet[2]{e}+\ve[3]{e}\vet[3]{e})\ma{S}(\ma{R}\ve[1]{e})\mtr,\label{eqB:correct}%\\
%={}&\ma{R}\ve[1]{e}\vet[1]{e}\mat{R}+\tfrac12(\ma{I}-\ve[1]{e}\vet[1]{e})-\tfrac12\ma{S}(\ma{R}\ve[1]{e})(\ma{I}-\ve[1]{e}\vet[1]{e})\ma{S}(\ma{R}\ve[1]{e}).
\end{align}
which is positive definite for all $\ma{R}\in\SOT$ by inspection.\end{proof}

\begin{remark}
The explicit solution to the linear system of equations \eqref{eqB:unique} is given  in \citep{markdahl2012exact}. The solution can also be obtained by solving the system using the Schur complement provided in the proof of Theorem \ref{thB:sol}. The explicit expression for the exact solutions, which is somewhat long and complicated, is omitted.
\end{remark}

\section{Numerical Example}

\noindent To provide an intuitive understanding for the workings of Algorithm \ref{algo:2014} let us consider an example of its behavior in simulation. The system trajectory in the case of
\begin{align*}
\ma{P}=\begin{bmatrix}
0 & 0 & 0\\
0 & 1 & 0\\
0 & 0 & 0
\end{bmatrix},\quad
\ma{R}(0)=\ma[0]{R}=\begin{bmatrix}
\phantom{-}0 & \phantom{-}\frac{1}{\sqrt3} & -\frac{2}{\sqrt6}\\
\phantom{-}\frac{1}{\sqrt2} & -\frac{1}{\sqrt3} & -\frac{1}{\sqrt{6}}\\
-\frac{1}{\sqrt{2}} & -\frac{1}{\sqrt{3}} & -\frac{1}{\sqrt{6}}
\end{bmatrix}
\end{align*}
is displayed in Figure \ref{fig:numerical1}. Observe that the reduced attitude corresponding to the diagonal path in Figure \ref{fig:numerical1} moves along a great circle on the unit sphere whereas the other two paths are non-geodesic. This is also clear from Figure \ref{fig:convergence1}; the shortest travelled distance equals the corresponding initial geodesic distance. The geodesic distance from $\ma{R}\ve[i]{e}$ to $\ve[i]{e}$ is given by $\arccos\ma[ii]{R}$, as can be shown by taking the inner product of the two vectors. Note that although $\arccos\ma[22]{R}(0)=\max_{i\in[3]}\arccos\ma[ii]{R}(0)$, it is $\arccos\ma[22]{R}(t)$ that converges to zero the fastest initially, see Figure \ref{fig:convergence1}.

\begin{figure}[htb!]
\centering
\includegraphics[width=0.9\textwidth]{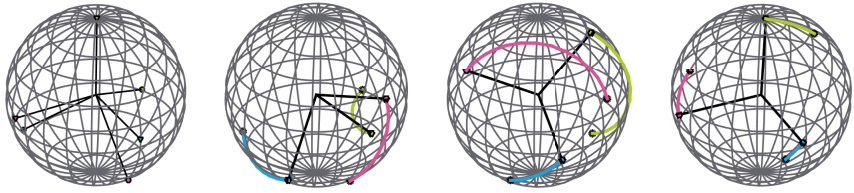}%traj_c41_v3.eps}
\caption{\label{fig:numerical1}The initial frame $\ma[0]{R}$ (circles) and desired frame $\ma{I}$ (arrows). A display of the system evolution on the time intervals $[0,1.2]$, $[1.2,2.4]$, and $[2.4,3.9]$ (left to right).}
\end{figure}
\begin{figure}[htb!]
\centering
\psfrag{t}{${\scriptstyle t}$}
\includegraphics[width=0.6\textwidth]{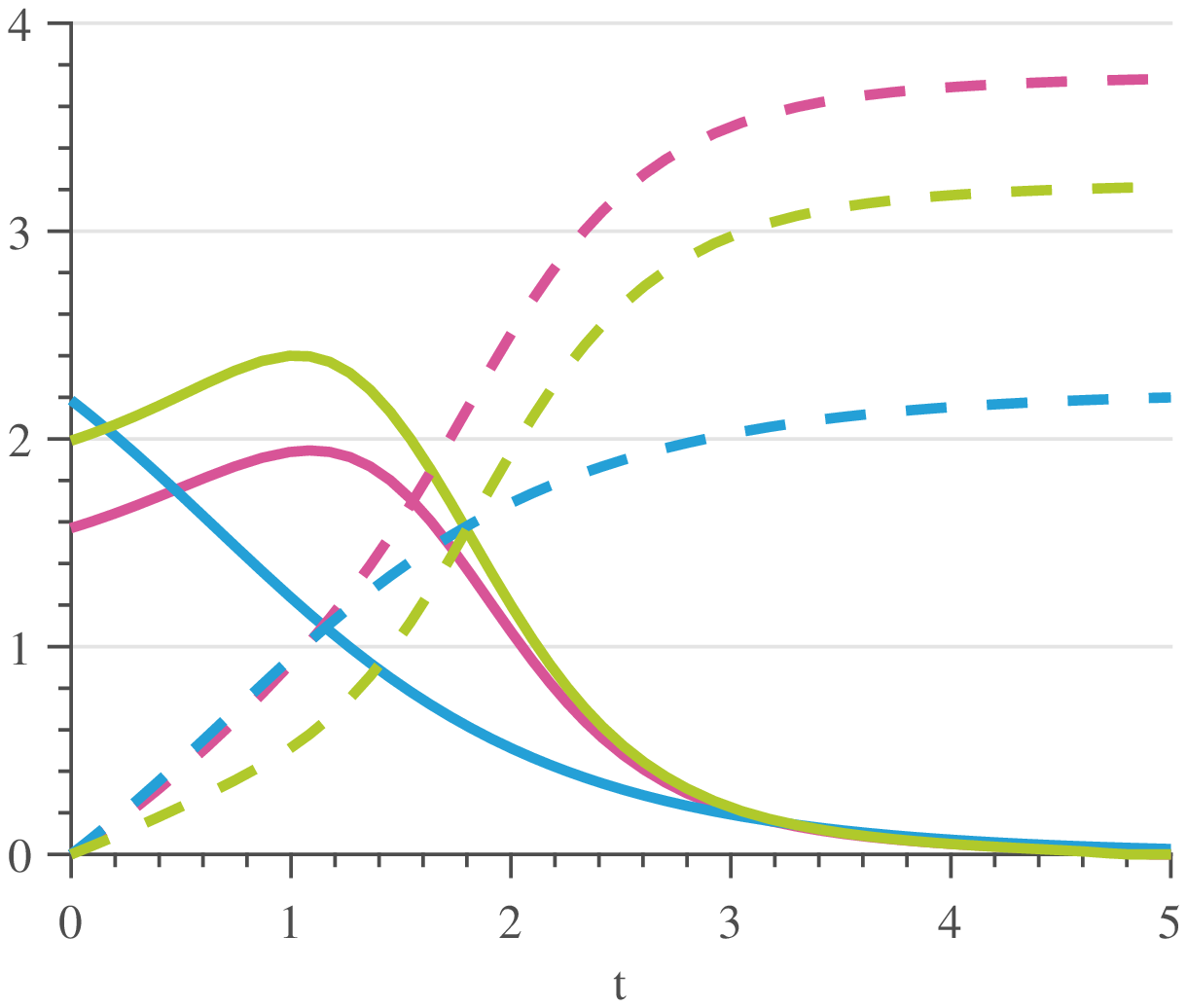}
% but psfrag below includegraphics
\psfrag{0}{${\scriptstyle 0}$}
\psfrag{1}{${\scriptstyle 1}$}
\psfrag{2}{${\scriptstyle 2}$}
\psfrag{3}{${\scriptstyle 3}$}
\psfrag{4}{${\scriptstyle 4}$}
\psfrag{5}{${\scriptstyle 5}$}
\caption{\label{fig:convergence1}The errors $\arccos(\ma[ii]{R})$ (solid lines) and travelled distance $\int_0^t\|\md{R}\ve[i]{e}\|_2\diff\tau$ (dashed lines).}
\end{figure}

\section{Conclusions}

\noindent This paper begins with the optimal control problem on $\SOT$ of minimizing the distance traveled by the reduced attitude while stabilizing the full attitude. Consider a two-step control sequence: first stabilize the reduced attitude and then align the remaining two vectors by means of a planar rotation. Its use would be inadvisable in practice due to a lack of either smoothness or precision; the first step must either display finite time convergence or a steady-state error. Rather, this paper fuses the two steps into one smooth motion. Being just a weighted sum of the two sequential control laws---it remarkably achieves almost global exponential stability. The setting is generalized to $\SO$ where geometric control techniques allow us to prove almost global exponential stability for a class of feedback laws that use orthogonal projection matrices as gain factors. Throughout the paper, we  contrasts the cases of $\SOT$ and $\SO$ with each other. Working with rotation matrices directly on the manifolds rather than in Euclidean space by means of parametrizations makes generalizations from $\SOT$ to $\SO$ come naturally, showcasing the strengths of the geometric control approach. %There are however certain algebraic relations on $\SOT$ that fail to hold in higher dimensions. For example, exact solutions to the closed-loop kinematics are provided explicitly for the $\SOT$ case, but these partial solutions are provided for the $\SO$ case. 

%The proof of stability in the general case by means of a two-step reduction theorem mimics our intuition for the $\SOT$ case. The invariance of the set of rotation matrices with non-negative spectrum is also interesting. It can be conjectured that such an invariance is generated by all continuous, time-invariant, almost globally stabilizing feedback laws on $\SO$.

\section{Acknowledgements}

\noindent  The authors gratefully acknowledges constructive and insightful feedback received from the anonymous reviewers.

\section{References}

\bibliographystyle{plain}
\bibliography{systemcontrolletters}

\renewcommand{\thesection}{\Alph{section}}
\setcounter{section}{0}
%\appendix
\section{Lemmas}  
\label{appB:lemmas}

\begin{lemma}\label{leB:unique}
The closed-loop system generated by Algorithm \ref{algo:2014} and system \eqref{sys:block} have unique solutions that belong to $\SO$ for all $t\in[0,\infty)$.
\end{lemma}

\begin{proof}
The kinematics \eqref{eqB:SOn} constrains the solutions to lie in $\SO$ for any initial condition on $\SO$ by restricting the instantaneous movement to $\ts[\SO]{\ma{I}}=\so$. Recall that it suffices to prove that the right-hand side is locally Lipschitz in $\ma{R}$ for all $\ma{R}\in\SO$ to establish global existence and uniqueness of solutions to \eqref{eq:closed2014} and system \eqref{sys:block} due to $\SO$ being a compact and invariant subset of $\R^{n\times n}$ \citep{khalil2002nonlinear}. Furthermore, any linear combination or product of two functions that are Lipschitz on a domain is also Lipschitz on the same domain. It is clear that the two right-hand sides can be decomposed in this manner using functions that are Lipschitz on $\SO$.
\end{proof}

\begin{lemma}\label{le:PRP} Let $\ma{P}$ be an orthogonal projection and $\ma{R}\in\SO$. Then $-1\notin\sigma(\ma{R})$ implies $-1\notin\sigma(\ma{P}\ma{○R}\ma{P})$. 
\end{lemma}

\begin{proof}
Suppose  $(-1,\ve{v})$ is an eigenpair of $\ma{P}\ma{R}\ma{P}$ with $\|\ve{v}\|_2=1$. Then $(\ma{P}\ve{v})^*\ma{R} \ma{P}\ve{v}\allowbreak=-1$. Since $\|\ma{R}\|_2=1$, this implies $\ma{P}\ve{v}=\ve{v}$ whereby $\ve{v}^*\ma{R}\ve{v}=-1$. The last identity requires that $(-1,\ve{v})$ is an eigenpair of $\ma{R}$.\end{proof}

%Since $\ma{P}$ is symmetric, the spectral theorem gives $\ma{P}=\ma{Q}\ma{\Lambda}\mat{Q}$ for some $\ma{Q}\in\{\ma{Q}\in\GL\,|\,\ma{Q}\inv=\mat{Q}\}$ and some diagonal matrix $\ma{\Lambda}$ with $\ma[ii]{\Lambda}\in\{0,1\}$. Suppose  $(-1,\ve{v})$ is an eigenpair of $\ma{P}\ma{R}\ma{P}$ with $\|\ve{v}\|_2=1$. Then
%$(\ma{\Lambda}\mat{Q}\ve{v})\mtr\mat{Q}\ma{R}\ma{Q}\ma{\Lambda}\mat{Q}\ve{v}=-1$. Since $\|\ma{\Lambda}\mat{Q}\ve{v}\|\|\mat{Q}\allowbreak\ma{R}\ma{Q}\|\|\ma{\Lambda}\mat{Q}\ve{v}\|\geq1$ and $\|\ma{\Lambda}\mat{Q}\ve{v}\|\leq1$, $\|\mat{Q}\ma{R}\ma{Q}\|=1$, it follows that $\|\ma{\Lambda}\mat{Q}\ve{v}\|=1$, \ie $\ma{\Lambda}\mat{Q}\ve{v}=\mat{Q}\ve{v}$. This implies that $\ma{R}=\ma{P}\ma{R}\ma{P}\in\NO$.

%By way of contradiction, suppose that $(-1,\ve{v})$ is an eigenpair of $\ma{P}\ma[0]{R}\ma{P}$. Note that
%%
%\begin{align*}
%\|\ma{P}\ma[0]{R}\ma{P}\ve{v}\|&\leq\|\ma{P}\|\|\ma{R}\|\|\ma{P}\ve{v}\|=\|\ma{P}\ve{v}\|=1.
%\end{align*}
%%
%From $\sigma(\ma{P})=\{0,1\}$, it follows that $(1,v)$ is an eigenpair of $\ma{P}$. So $\ma{P}\ma[0]{R}\ve{v}=-\ve{v}$. But this implies that $\ma{R}\ve{v}=-\ve{v}$, which contradicts the assumption of $-1\notin\sigma(\ma{R})$.

\begin{lemma}\label{eqB:relations}
Consider the block-matrix partitions 
\begin{align*}
\ma{R}=\begin{bmatrix}
\ma[11]{R} & \ve[12]{R}\\
\ve[21]{R} & \ma[22]{R}
\end{bmatrix}\in\SO,\quad
\ma{R}=\begin{bmatrix}
r_{11} & \ve[12]{r}\\
\ve[21]{r} & \ma[22]{R}
\end{bmatrix}\in\SOT,
\end{align*}
where $\ma[22]{R}\in\R^{2\times2}$. The relations $\ve[21]{r}\ve[12]{r}=r_{11}\ma[22]{R}+\ma{S}\ma[22]{R}\ma{S}$  and $(\trace\ma[22]{R})^2+(\trace\ma[22]{R}\ma{S})^2=(1+r_{11})^2$, where 
\begin{align*}
\ma{S}\in\left\{\begin{bmatrix}
0 & -1\\
1 & \phantom{-}0
\end{bmatrix},\begin{bmatrix}
\phantom{-}0 & 1\\
-1 & 0
\end{bmatrix}\right\},
\end{align*}
holds on $\SOT$. The relation $\ma[22]{R}(\mat[22]{R}-\ma[22]{R})=\trace(\ma[22]{R}\ma{S})\ma[22]{R}\ma{S}$ holds on $\SO$.
\end{lemma}

\begin{proof}
The proof is by elementary calculations on the level of matrix elements. We only provide a partial proof. The first relation is given by
\begin{align*}
\begin{bmatrix}
r_{21}r_{12} & r_{21}r_{13}\\
r_{31}r_{13} & r_{31}r_{13}
\end{bmatrix}=\begin{bmatrix}
r_{11}r_{22} & r_{11}r_{23}\\
r_{11}r_{32} & r_{11}r_{33}
\end{bmatrix}+\begin{bmatrix}
-r_{33} & \phantom{-}r_{32}\\
\phantom{-}r_{23} & -r_{22}
\end{bmatrix},
\end{align*}
where $\ma{R}=[r_{ij}]$. The north-west of these four identities states that $r_{33}=r_{11}r_{22}-r_{21}r_{12}$ which follows from setting the cross product of the first two columns in $\ma{R}$ equal to the third. The other identities can be proven to hold by reasoning analogously.
\end{proof}

%% The Appendices part is started with the command \appendix;
%% appendix sections are then done as normal sections
%% \appendix

%% \section{}
%% \label{}

%% If you have bibdatabase file and want bibtex to generate the
%% bibitems, please use
%%
%%  \bibliographystyle{elsarticle-num} 
%%  \bibliography{<your bibdatabase>}

%% else use the following coding to input the bibitems directly in the
%% TeX file.

%\section{To Do}
%
%Name the norms.\\
%Eigenpairs $\lambda\ve{v}$ and $\mu\ve{w}$?\\
%Should be able to get $k$ in general stability of identity proof. Perhaps not in almost global?
%If almost global, then change/remove proof in SO(3) section
%Use $\mathcal{S}(\C)$ to denote the unit circle.

\end{document}

%% file: myIncludes.tex
\usepackage{subdepth} % fixes issues with superscript and subscript vertical position when both are used (there is still some issue with horizontal placement, look at $\|x\|_2^2$)

\usepackage{graphics} % for pdf, bitmapped graphics files
\usepackage{epsfig} % for postscript graphics files
\usepackage{mathrsfs} % for generator script

\usepackage{amsmath,euscript,psfrag,color,graphicx}
\usepackage{amssymb,amsthm, amsfonts}
\usepackage{graphicx}
\usepackage{latexsym}

\usepackage[small,hang,bf,sf,up,labelsep=period]{caption} % have not checked that it works, no figures yet
\usepackage{quoting} % for abstracts
\usepackage{lettrine} % for the first letter
\usepackage{tocloft} % for the table of contents title
\usepackage{titletoc}
\usepackage{fmtcount}
\usepackage{calc} % to use arithmics with counters
\usepackage{cite} 
\usepackage{multicol} % this could be used to create two column chapters
\usepackage{centernot}
\usepackage{nicefrac}
\usepackage{dsfont}

\usepackage{lipsum}
\usepackage{psfrag}
\usepackage[export]{adjustbox} %\adjincludegraphics is more powerful than includegraphics
\PassOptionsToPackage{cmyk}{xcolor}
\usepackage[cmyk]{xcolor}
\selectcolormodel{cmyk}

\usepackage{bm}
\usepackage{textcomp}

\newcommand{\ve}[2][]{\ensuremath{\boldsymbol{\mathrm{#2}}}_{#1}}
\newcommand{\vet}[2][]{\ensuremath{\boldsymbol{\mathrm{#2}}^{\top}_{#1}}}
\newcommand{\vd}[2][]{\ensuremath{\dot{\boldsymbol{\mathrm{#2}}}_{#1}}}

\newcommand{\ma}[2][]{\ensuremath{\boldsymbol{\mathrm{#2}}}_{#1}}
\newcommand{\mat}[2][]{\ensuremath{\boldsymbol{\mathrm{#2}}^{\top}_{#1}}}
\newcommand{\md}[2][]{\ensuremath{\dot{\boldsymbol{\mathrm{#2}}}}_{#1}}
\newcommand{\mdt}[2][]{\ensuremath{\dot{\boldsymbol{\mathrm{#2}}}^{\top}_{#1}}}

\newcommand{\R}{\ensuremath{\mathds{R}}}

\newcommand{\C}{\ensuremath{\mathds{C}}}

\newcommand{\N}{\ensuremath{\mathds{N}}}

\newcommand{\Ham}{\ensuremath{\mathds{H}}}

\newcommand{\SOT}{\ensuremath{\mathsf{SO}(3)}}

\newcommand{\SO}{\ensuremath{\mathsf{SO}(n)}}

\newcommand{\sot}{\ensuremath{\mathsf{so}(3)}}

\newcommand{\so}{\ensuremath{\mathsf{so}(n)}}

\newcommand{\SET}{\ensuremath{\mathsf{SE}(3)}}

\newcommand{\GL}{\ensuremath{\mathsf{GL}(n)}}

\newcommand{\NO}{\ensuremath{\mathcal{N}}}

\newcommand{\Sn}{\ensuremath{\mathcal{S}^{n}}}

\newcommand{\St}{\ensuremath{\mathcal{S}^{2}}}

\newcommand{\GC}{\ensuremath{\mathcal{S}^1}}

\newcommand{\ie}{\textit{i.e.}, }

\newcommand{\eg}{\textit{e.g.}, }

\newcommand{\mtr}{\hspace{-0.3mm}\ensuremath{^\top}}

\newcommand{\diff}{\ensuremath{\mathrm{d}}}

\newcommand{\inv}{\ensuremath{^{-1}}}

\newcommand\raiseT[2]{\setbox0\hbox{$#1{#2}$}\raise\dp0\box0}

\newcommand{\etc}{\emph{etc.\ }}

\DeclareMathOperator{\e}{\ensuremath{\mathrm{e}}}

\DeclareMathOperator{\Log}{\ensuremath{\mathrm{Log}}}

\DeclareMathOperator{\sech}{\ensuremath{\mathrm{sech}}}

\DeclareMathOperator{\Atanh}{\ensuremath{\mathrm{Atanh}}}

\DeclareMathOperator*{\argmax}{\text{argmax}}

\DeclareMathOperator*{\argmin}{\text{argmin}}

\DeclareMathOperator{\rank}{\ensuremath{\mathrm{rank}}}

\DeclareMathOperator{\trace}{\ensuremath{\mathrm{tr}}}

\makeatletter
\newcommand{\raisemath}[1]{\mathpalette{\raisem@th{#1}}}
\newcommand{\raisem@th}[3]{\raisebox{#1}{$#2#3$}}
\makeatother

\newcommand{\ts}[2][]{\ensuremath{\mathsf{T}_{#2}#1}}

\newcommand{\maspace}{\hspace{-2mm}}

\definecolor{kthbluergb}{RGB}{25,84,166}%{0,108,183}
\definecolor{kthbluecmyk}{cmyk}{1,0.55,0,0}

\definecolor{kthblueA}{RGB}{25,84,166}%{0,108,183}
\definecolor{kthblueB}{RGB}{46,124,192}%{0,108,183}
\definecolor{kthblueC}{RGB}{112,153,209}%{0,108,183}
\definecolor{kthblueD}{RGB}{164,186,225}%{0,108,183}
\definecolor{kthblueE}{RGB}{211,220,241}%{0,108,183}

\newcommand{\um}{\hspace{0.2mm}}

\newcounter{counter} % there is only one counter

\newtheoremstyle{myplain} % name
{\topsep}                    % Space above
{\topsep}                    % Space below
{\itshape}                   % Body font
{}                           % Indent amount
{\bfseries\sffamily}         % Theorem head font
{.}                          % Punctuation after theorem head
{.5em}                       % Space after theorem head
{}  % Theorem head spec (can be left empty, meaning ‘normal’)

\theoremstyle{myplain}
\newtheorem{theorem}[counter]{Theorem}
\newtheorem{lemma}[counter]{Lemma}
\newtheorem{proposition}[counter]{Proposition}

\newtheoremstyle{mydefinition} % name
{\topsep}                    % Space above
{\topsep}                    % Space below
{}                   % Body font
{}                           % Indent amount
{\bfseries\sffamily}         % Theorem head font
{.}                          % Punctuation after theorem head
{.5em}                       % Space after theorem head
{}  % Theorem head spec (can be left empty, meaning ‘normal’)

\theoremstyle{mydefinition}

\newtheorem{example}[counter]{Example}

\newtheorem{algorithm}[counter]{Algorithm}

\newtheoremstyle{myremark} % name
{\topsep}                    % Space above
{\topsep}                    % Space below
{}                   % Body font
{}                           % Indent amount
{\itshape}         % Theorem head font
{.}                          % Punctuation after theorem head
{.5em}                       % Space after theorem head
{\thmname{#1}\thmnumber{ #2}\thmnote{.#3}}  % Theorem head spec (can be left empty, meaning ‘normal’)

\theoremstyle{myremark}
\newtheorem{remark}[counter]{Remark}


\newcounter{parentnumber}